\documentclass[12pt]{article}
\usepackage{makeidx}
\usepackage{times}
\usepackage{amssymb}
\usepackage{amsmath}
\numberwithin{equation}{section}
\usepackage{setspace}
\usepackage{color}
\usepackage{graphicx}
\usepackage{rotating}
\usepackage{pdflscape}
\usepackage{epstopdf}
\usepackage[round,authoryear,comma]{natbib}
\usepackage{natbib,hyperref}
\usepackage{booktabs}
\usepackage{setspace}
\usepackage{pgfplots}
\usepackage{geometry}
\usepackage{subcaption}
\usepackage{xcolor}
\usepackage{multirow}
\usepackage{array}

\hypersetup{
	colorlinks   = true, 
	urlcolor     = blue, 
	linkcolor    = blue, 
	citecolor   = blue 
}
\providecommand{\U}[1]{\protect\rule{.1in}{.1in}}
\newtheorem{theorem}{Theorem}[section]

\newtheorem{corollary}{Corollary}

\newtheorem{definition}{Definition}[section]

\newtheorem{lemma}{Lemma}[section]

\newenvironment{proof}[1][Proof]{\noindent \textbf{#1.} }{\hfill
	\rule{0.5em}{0.5em}}

\parskip=\medskipamount

\newcolumntype{C}[1]{>{\centering\arraybackslash}m{#1}}

\begin{document}

\title{Value-at-Risk, Tail Value-at-Risk and upper tail transform of the sum of two counter-monotonic random variables}
\date{}
\author{
Hamza Hanbali \\ \small Monash University\\ \small \texttt{hamza.hanbali@monash.edu} \and 
 Dani\"{e}l Linders \\ \small University of Amsterdam. \\ \small \texttt{dlinders@illinois.edu} \and Jan Dhaene \\ \small KU Leuven \\ \small \texttt{jan.dhaene@kuleuven.be}}
\maketitle
\begin{abstract}
	The Value-at-Risk (VaR) of comonotonic sums can be decomposed into marginal VaR's at the same level. This additivity property allows to derive useful decompositions for other risk measures. In particular, the Tail Value-at-Risk (TVaR) and the upper tail transform of comonotonic sums can be written as the sum of their corresponding marginal risk measures.	

	The other extreme dependence situation, involving the sum of two arbitrary counter-monotonic random variables, presents a certain number of challenges. One of them is that it is not straightforward to express the VaR of a counter-monotonic sum in terms of the VaR's of the marginal components of the sum. This paper generalizes the results derived in \cite{Chaoubi} by providing decomposition formulas for the VaR, TVaR and the stop-loss transform of the sum of two arbitrary counter-monotonic random variables.
 
\end{abstract}
\newpage
\section{Introduction}
\label{Introduction} 
This paper provides decompositions for the Value-at-Risk (VaR), Tail Value-at-Risk (TVaR), and the upper tail transform (or stop-loss premium), of the sum of two counter-monotonic random variables with arbitrary marginal distributions. It is shown that the VaR can be expressed in terms of marginal quantiles but at different levels. Further, it is shown that the TVaR and the upper tail transform can be expressed in terms of linear combinations of the marginal risk measures, with additional terms required when the underlying random variables are discrete. The uniqueness of the representation of the VaR and the TVaR in terms of marginal risk measures is however not guaranteed in general.

The results of the present paper are useful in situations involving sums or differences of two random variables. These situations underlie a broad range of practical actuarial and financial problems. In an insurance context, asset and liability management requires the analysis of a difference of two random variables. The liability side of the balance sheet is in major part determined by actuarial risks, whereas the asset side, on the other hand, is a mirror image of the insurer's investment strategy \citep{Wuthrich}. Other situations include the evaluation of basket and spread options, whose payoffs are expressed as stop-loss premiums of sums, or differences, of random variables \citep{DhaeneGoovaerts,CarmonaDurrleman2003,LaurenceWang2008,LaurenceWang2009}. The complex instruments from the emerging market of longevity and mortality derivatives, such as longevity trend bonds or catastrophic mortality bonds, also involve differences of random variables \citep{HuntBlake2015,ChenMacMinnSun2015,BahlSabanis}, and so does the management of basis risk \citep{Coughlaan2011,Zhang2017}.

Determining the distribution and risk measures for linear combinations of random variables requires modeling their joint distribution. As pointed out in a number of earlier related studies (see e.g.\ \cite{Kaas_etal:Upper_Lower_Bounds}, \cite{BernardJBF,Bignozzi_et_al_2015}), modeling the individual risks, or marginal distributions, is often considered a standard task by actuaries and risk managers, whereas a more challenging task consists in determining their dependence structure. Indeed, an extensive statistical and quantitative toolbox is available to model separately actuarial risks, such as future mortality dynamics or claim counts in non-life insurance, and financial risks, such as future returns of an investment portfolio. However, the joint behavior of actuarial and financial risks is hard to estimate, and it is not uncommon that the problem is simplified by assuming an independence structure.

A reasonable risk management strategy consists in analyzing the most extreme situations for the joint behavior of a random vector with given marginal distributions. For sums of two random variables, these extreme situations are modeled by comonotonicity (i.e.\ extreme positive dependence) and counter-monotonicity (i.e.\ extreme negative dependence). In the framework of convex ordering, comonotonicity corresponds to the worst-case scenario and counter-monotonicity corresponds to the best-case scenario \citep{DenuitDhaeneGoovaertsKaas2005,ShakedShanthihumar,Cheung_Lo_2014}. Therefore, the TVaR and the upper tail transform of the random sum will reach their maximal and minimal value in case the random variables are comonotonic and counter-monotonic, respectively. Under the assumption of extreme dependence structures, the TVaR and the upper tail transform can be used to assess the diversification benefit from combining different risks or investment opportunities \citep{Embrechts2009,Embrechts2013b}. Further, since even the most hardened modelers are not exempt from model risk, it is prudent to use sharp comonotonic and counter-monotonic bounds which are consistent with the available information on the univariate distributions \citep{Dhaene2000,Kaas_etal:Upper_Lower_Bounds}. These bounds are also informative on the extent of dependence model risk, which can be quantified using the dependence uncertainty spread, i.e.\ the difference between the upper and lower bounds \citep{Embrechtal2015,Bignozzi_et_al_2015}. In particular, the knowledge about the marginal distributions does not unambiguously determine the value of a risk measure, since infinitely many values can be obtained from the spectrum of all possible dependence structures. Hence, the dependence uncertainty spread allows insurers and risk managers to evaluate their vulnerability to the choice of a dependence structure\footnote{A substantial amount of research has been devoted to derive alternative analytical and numerical worst- and best-case bounds. For the VaR, these bounds are necessary because comonotonic and counter-monotonic VaR's are not always worst- and best-case bounds. Some notable contributions in this direction are due to \cite{Mesfioui}, \cite{WangPengYang2013}, \cite{Embrechts2013b}, \cite{BernardJRI}, and more recently \cite{Luxa}. Approaches tailored for the TVaR or for distortion risk measures that incorporate additional information can be found in \cite{Kaas_etal:Upper_Lower_Bounds}, \cite{Puccetti2013} and \cite{Bernardetal2014}. Nevertheless, a trade-off arises from this discussion. Including more information would indeed tighten the bounds, but it could also add to the dependence model risk, because of the uncertainty that stems from that information. In contrast, the comonotonic and counter-monotonic bounds provide more safety in the analysis of the dependence uncertainty spread, at the cost of a wider interval.}.

The decompositions of the TVaR and the stop-loss premium of comonotonic sums are standard results in actuarial science \citep{MeileisonNadas1979,Dhaene2000,Dhaene_etal_2006,Hobson2005,Chen2008}. A key property underpinning these results is the additivity of quantiles of comonotonic sums, which is itself due to the non-decreasing properties of quantiles \citep{Dhaene2002a,McNeil}. Solving the corresponding problem involving extreme negative dependence is more challenging. In the bivariate case\footnote{Characterizing extreme negative dependence for more than two random variables is still an active area of research. One of the proposed notions is complete mixability, which minimizes the variance of sums of random variables and coincides with counter-monotonicity in the bivariate case; see \cite{GaffkeRuschendorf} for the foundational principles and \cite{Wang_wANG_2011cm,WangPengYang2013} and references therein for additional results pertaining to VaR minimization. \cite{Dhaene_Denuit_1999} proposed mutual exclusivity, which leads to convex lower bounds for some classes of multivariate distributions with given marginals. \cite{Cheung_Lo_2014} provide details on the intrinsic relationship between pairwise counter-monotonicity and multivariate mutual exclusivity. Other proposals for multivariate extreme negative dependence include \cite{LeeAhn_2014}'s $d$-counter-monotonicity and \cite{PuccettiWang}'s $\Sigma$-counter-monotonicity; see also \cite{WangWang2015,WangWang2016}. Nevertheless, this fertile literature has scarcely addressed the decomposition of the TVaR and the stop-loss premiums. A noteworthy result is due to \cite{Dhaene_Denuit_1999} who proved the additivity of stop-loss premiums for sums of mutually exclusive risks under some conditions on the marginal distributions; see also \cite{CheungLo2013} for other convex functionals of sums, as well as \cite{Cheung2017} for tail mutual exclusivity.}, which is of interest to the present paper, deriving explicit decomposition formulas for the VaR, TVaR, and stop-loss premium of the sum of counter-monotonic random variables with arbitrary marginal distributions remains an open problem. 

For the equivalent problem of comonotonic differences, \cite{LaurenceWang2009} derive a decomposition for stop-loss premiums in the case of two continuous random variables, but they do not provide closed-form expressions for the corresponding retentions. \cite{Chaoubi} also restrict their analysis to continuous random variables with a specific choice of marginal distributions, including symmetric or unimodal distributions. When the random variables follow such marginal distributions, the function defined as the sum of quantiles of counter-monotonic random variables have appealing properties. In particular, it is continuous, and either strictly monotone or globally convex/concave. However, this paper shows that many situations involve functions with complex shapes which are not covered in the existing literature. A typical case is when one of the random variables is discrete (e.g.\ number of survivors) whereas the other is continuous (e.g.\ investment fund value). Another one is when the marginal distributions have fat tails or are skewed. In such situations, the existing decompositions are not valid. \cite{HanbaliLinders} provide decompositions for the VaR of comonotonic differences under more general assumptions using the concept of tail monotone functions, where the decomposition holds for certain values of the quantile level. Their results can readily be applied to decompose the TVaR as well, but only at confidence levels beyond some thresholds.

The present paper contributes to the body of literature devoted to the decomposition of risk measures of sums of two random variables \citep{Dhaene_Denuit_1999,CheungLo2013,Cheung2017}. The paper extends the work of \cite{LaurenceWang2009} and \cite{Chaoubi} by deriving explicit decompositions of the VaR, TVaR and stop-loss premium of counter-monotonic sums with general marginal distributions. Whereas the decomposition of the VaR is a natural result that follows from the basic definitions, the decomposition of TVaR's and stop-loss premiums are more challenging. This is due to the fact that the sum of quantiles of two counter-monotonic random variables may exhibit complex features, such as jumps (and hence, does not fall in the framework considered in \cite{LaurenceWang2009} and \cite{Chaoubi} with continuous random variables) or multiple inflection points (and hence, for VaR and TVaR, violates \cite{Chaoubi}'s convex/concave assumptions).

A standard result in the literature is that the quantile of a random variable transformed by a monotone function can be expressed as the quantile of the original function transformed using that function \citep{Dhaene2002a}. For instance, for a continuous and non-decreasing function $g$ and a random variable $X$, it holds that $F^{-1}_{g(X)}(p)=g\left(F^{-1}_{X}(p)\right)$ for all $p\in(0,1)$. The situation where $g$ is not monotone is problematic. \cite{HanbaliLinders} address the problem and show that the equality $F^{-1}_{g(X)}(p)=g\left(F^{-1}_{X}(p)\right)$ holds for specific values of $p$ if the function $g$ has a monotone tail. The present paper adds to their insights by showing that if $g$ is not monotone, then that equality does not necessarily hold. Further, when it does hold, there does not always exist a unique $p^{\star}$ such that $F^{-1}_{g(X)}(p)=g\left(F^{-1}_{X}(p^\star)\right)$.

The remainder of the paper is organized as follows. Section \ref{Sec:Preliminaries} introduces the notations and relevant definitions. A definition of the companion comonotonicity concept is also provided, as well as the analogous results available in the literature. Section \ref{Sec:Countermonotonicity} contains a formal definition of counter-monotonicity. The main problem of the paper is exposed, and some key elements for the subsequent derivations are provided, together with some numerical examples. The main results of the paper are derived in Sections \ref{Sec:VaR}, \ref{Sec:TVaR} and \ref{Sec:SL}, dealing with the decompositions of the VaR, TVaR and the stop-loss premium, respectively. Each section contains the main results, followed by a discussion and numerical illustrations. Section \ref{Sec:Conclusion} concludes the paper. In order to ease the reading, all proofs are relegated to the appendix.

\section{Preliminaries} \label{Sec:Preliminaries}
\subsection{Notations}
All random variables are defined on a common probability space $\left(\Omega,\mathcal{F},
\mathbb{P} \right)$. They can be discrete, continuous, or a combination of the two. It is assumed that all random variables are such that the risk measures introduced hereafter are finite. 

The cumulative distribution function (cdf) of a random variable
 $X$ is denoted by $F_X$. The left inverse of $F_X$ is denoted by $F_X^{-1}$, and is defined as $
F_{X}^{-1}\left(  p\right)  =\inf\left\{  x\in\mathbb{R}\mid F_{X}(x)\geq
p\right\}$, for $p\in\left[  0,1\right],$ with $\inf\emptyset=+\infty$ by convention. The Value-at-Risk (VaR) of a random variable $X$ at the level $p\in[0,1]$ is given by $\text{VaR}_p[X]=F^{-1}_X(p)$. The right inverse of $F_X$ is denoted by $F_X^{-1+}$, and is defined as $
F_{X}^{-1+}\left(  p\right)  =\sup\left\{  x\in\mathbb{R}\mid F_{X}(x)\leq
p\right\},$ for $p\in\left[  0,1\right],$ with $\sup\emptyset=-\infty$. If \textbf{$F_X$ is continuous and strictly increasing}, the left and right inverses are equal. Otherwise, on horizontal segments of $F_X$, the inverses $F_{X}^{-1}\left( p\right)$ and $F_{X}^{-1+}\left(
p\right)$ are different. In such cases, the generalized $\alpha$-inverse $F_{X}^{-1\left(  \alpha\right) }$, for any $\alpha\in\left[  0,1\right]  $, is defined as follows:
\begin{equation} \label{Eq_2_1}
F_{X}^{-1\left(  \alpha\right)  }\left(  p\right)  =(1-\alpha) F_{X}^{-1}\left(
p\right)  +  \alpha F_{X}^{-1+}\left( p\right),\qquad p\in\left[  0,1\right]. 
\end{equation}
More details on generalized inverses can be found in \cite{Dhaene2002a} and \cite{Embrechts2013a}, among others. 

The Tail-Value-at-Risk (TVaR) at the level $p\in[0,1]$ of $X$ is defined as follows:
\begin{equation}\label{Eq_2_2}
  \text{TVaR}_p[X]=\frac{1}{1-p}\int_{p}^{1}F_X^{-1}(q)\text{d}q,
\end{equation}
and the Lower Tail-Value-at-Risk (LTVaR) at the level $p\in[0,1]$ of $X$ is defined as follows:
\begin{equation}\label{Eq_2_3}
  \text{LTVaR}_p[X]=\frac{1}{p}\int^{p}_{0}F_X^{-1}(q)\text{d}q.
\end{equation}

The upper tail transform of the random variable $X$ at the level $x\in \mathbb{R}$, also called the stop-loss premium with retention $x$, is denoted by $\pi_X(x)$, and is given by the following expected value:
\begin{equation}\label{Eq_2_4}
\pi_X(x)=\mathbb{E}\left[\left(X-x\right)_+\right] = \int_x^{+\infty}\left(1-F_{X}(y)\right)\text{d}y,
\end{equation}
where $(x)_+=\max\{x,0\}$. The lower tail transform of $X$ at the level $x\in\mathbb{R}$ is denoted by $\lambda_X(x)$, and is given by the following expected value:
\begin{equation}\label{Eq_2_5}
\lambda_X(x)=\mathbb{E}\left[\left(x-X\right)_+\right] = \int_{-\infty}^x F_{X}(y)\text{d}y.
\end{equation}
The upper and lower tail transforms $\pi_X(x)$ and $\lambda_X(x)$ are linked through the following relation:
\begin{equation}\label{Eq_2_6}
\pi_X(x) - \lambda_X(x) = \mathbb{E}[X] - x,
\end{equation}
which is referred to as the \textit{put-call parity} in quantitative finance.

For any $p\in(0,1)$ and $\alpha \in [0,1]$, the following relation establishes a link between the TVaR and the upper tail transform:
\begin{equation}\label{Eq_2_7}
\text{TVaR}_p[X] = F^{-1(\alpha)}_{X}(p) + \frac{1}{1-p}\pi_X\left(F^{-1(\alpha)}_X(p)\right),
\end{equation}
which can be rearranged to obtained the link between the LTVaR and the lower tail transform:
\begin{equation}\label{Eq_2_7x}
	\text{LTVaR}_p[X] = F^{-1(\alpha)}_{X}(p) - \frac{1}{p}\lambda_X\left(F^{-1(\alpha)}_X(p)\right).
\end{equation}
The proof of \eqref{Eq_2_7} can be found in e.g. \cite{Dhaene_etal_2006}.

Consider the bivariate vector $\left(X_1,X_2\right)$ and the sum $S$ given by $S=X_1+X_2$. The goal is to derive decomposition formulas for the VaR, TVaR, and stop-loss premium of the sum $S$, where the components of the vector $\left(X_1,X_2\right)$ are counter-monotonic, i.e.\ extreme negative dependence. In order to better understand the challenges when working with counter-monotonic random variables, it is helpful to shortly revisit the analogous results for comonotonicity. This is the focus of the remainder of this section.
%
%

\subsection{Comonotonic sums}
\begin{definition}[Comonotonic modification]
The \emph{comonotonic modification} of a bivariate random vector $\underline{X}$ is denoted by $\underline {X}^{c}=\left(
X_{1}^{c},X_{2}^{c}  \right)$, and is given by:
\begin{equation}\label{Eq_2_8}
\underline{X}^{c}\overset{d}{=}\left(  F_{X_{1}}^{-1}\left(  U\right)
,F_{X_{2}}^{-1}\left(  U\right)  \right)  ,
\end{equation}
where $\overset{d}{=}$ stands for equality in
distribution, and $U$ is uniformly distributed over the unit interval.
\end{definition}
Characterization (\ref{Eq_2_8}) shows that the components of the comonotonic random vector $\underline{X}^c$ are jointly driven by a single random source transformed by the non-decreasing functions $F_{X_{1}}^{-1}$ and $F_{X_{2}}^{-1}$. The marginals of the comonotonic modification remain, however, equal in distribution to those of the original vector $\underline{X}$. Comonotonicity was extensively discussed in \cite{Dhaene2002a}, and its applications in finance and actuarial science are numerous; see e.g.\\ \cite{Dhaene2002b}, \cite{Dhaene_et_al_2005}, \cite{DenuitDhaene2007}, \cite{Deelstra2010:overview_Comonotonicity}, \cite{FengJingDhaene} and \cite{HanbaliLinders2019}. 

Let $S^u$ be the sum of the components of $\underline{X}^c$. Then $S^u \overset{d}{=} F^{-1}_{X_1}(U)+F^{-1}_{X_2}(U)$. 
For any dependence structure of the random vector $\left(X_1,X_2\right)$, both the TVaR and the stop-loss premium of the sum $S=X_1+X_2$ are bounded from above by their comonotonic counterparts:
\begin{equation}\label{Eq_2_9}
\begin{array}{rll}
\text{TVaR}_p[S]&\leq \text{TVaR}_p[S^u],&\qquad \text{for all } p\in [0,1],\\
\pi_S(x)&\leq \pi_{S^u}(x),&\qquad \text{for all } x \in \mathbb{R}.
\end{array}\end{equation}
The proofs of \eqref{Eq_2_9} can be found in \cite{Dhaene2000} and \cite{Kaas_etal:Upper_Lower_Bounds}.

Quantiles are non-decreasing functions, and hence, the function $f:u\mapsto F^{-1}_{X_1}(u)+F^{-1}_{X_2}(u)$ is also non-decreasing. It is well-known that the quantile of a random variable transformed by a monotone function can be expressed as the quantile of the original random variable transformed using that function \citep{Dhaene2002a}. Thus, combining $S^u\overset{d}{=}f(U)$ and the non-decreasingess of $f$, the following decomposition of the Value-at-Risk holds:
$$\text{VaR}_p\left[S^u\right]=f(p)=\text{VaR}_p\left[X_1\right]+\text{VaR}_p\left[X_2\right],\qquad \text{for all } p\in[0,1].$$
Note that this decomposition is also valid when $S^u$ is a sum of more than two random variables; see e.g.\ \cite{Dhaene2002a}. Taking into account \eqref{Eq_2_2}, it holds that:
\begin{equation}\label{Eq_2_10}\text{TVaR}_p\left[S^u\right]=\text{TVaR}_p\left[X_1\right] + \text{TVaR}_p\left[X_2\right],\qquad \text{for all } p\in[0,1].\end{equation}
This decomposition is also not limited to the two-dimensional case; see \cite{Dhaene_etal_2006}.

For any $x\in\left(F^{-1}_{S^u}(0),F^{-1+}_{S^u}(1)\right)$, the comonotonic stop-loss premium $\pi_{S^u}(x)$ admits the
decomposition $\pi_{S^u}(x) = \pi_{X_1}(x_1) + \pi_{X_2}(x_2)$, where
\begin{equation}\label{Eq_2_11}x_i=F_{X_i}^{-1(\alpha_x)}(F_{S^{u}}(x)),\end{equation}
for $i=1,2$, and $\alpha_x$ follows from $x_1+x_2=x$.

The fact that the stop-loss premium of a comonotonic sum can be expressed as a sum of the marginal stop-loss premiums is a long-established result (see \cite{MeileisonNadas1979}), and is valid regardless of the number of components of the comonotonic sum $S^{u}$. The explicit expression \eqref{Eq_2_11} for the retentions $x_i$ was derived in \cite{Dhaene2000};
see also \cite{Hobson2005}, \cite{Chen2008}, \cite{Chen2015} and \cite{Linders2012}. 

\section{Counter-monotonic sums}\label{Sec:Countermonotonicity}
	\subsection{Definition}
\begin{definition}[Counter-monotonic modification]

	The \emph{counter-monotonic modification} of a bivariate random vector $\underline{X}$ is denoted by $\underline{X}^{cm}=\left(  X_{1}^{cm},X_{2}^{cm}  \right)$, and is given by:
	$$
	\underline{X}^{cm}\overset{d}{=}\left(  F_{X_{1}}^{-1}\left(  U\right)
	,F_{X_{2}}^{-1}\left(1-  U\right)  \right)  ,
	$$
	where $\overset{d}{=}$ stands for\textit{ equality in
		distribution}, and $U$ is uniformly distributed over the unit interval.
\end{definition}

The counter-monotonic modification $\underline{X}^{cm}$ of a random vector $\underline{X}$ has the same marginal distributions as $\underline{X}$, but the components of $\underline{X}^{cm}$ move perfectly in the opposite direction. Hence, counter-monotonicity corresponds to extreme negative dependence structure. 

Counter-monotonicity leads to a convex lower bound for a sum of dependent risks; see e.g.\\ \cite{Ruschendorf1983}, \cite{Dhaene_Denuit_1999} and \cite{Cheung_Lo_2014}. Let $S^l$ be the sum of the components of the transformed vector $\underline{X}^{cm}$, such that $S^l = F^{-1}_{X_1}(U)+F^{-1}_{X_2}(1-U)$.

Analogously to the inequalities in \eqref{Eq_2_9}, for any dependence structure of the random vector $\left(X_1,X_2\right)$, both the TVaR and the stop-loss premium of the sum $S=X_1+X_2$ are bounded from below by their counter-monotonic counterparts:
$$
\begin{array}{rll}
\text{TVaR}_p[S^l]&\leq \text{TVaR}_p[S],&\qquad \text{for all } p\in [0,1],\\
\pi_{S^l}(x)&\leq \pi_{S}(x),&\qquad \text{for all } x \in \mathbb{R}.
\end{array}
$$
But the analogy with comonotonic sums breaks down at this point. In particular, the decomposition of the VaR, TVaR and stop-loss premium of the counter-monotonic sum of arbitrary random variables is still an open problem, to which the present paper is devoted. 

In order to understand the challenge, recall the function $f$:
\begin{equation}\label{Eq_3_1}
f(u) = F^{-1}_{X_1}(u) + F^{-1}_{X_2}(u),\qquad \text{for } u\in(0,1),
\end{equation}
and consider the following function $g$:
\begin{equation}\label{Eq_3_2}
g(u)=F_{X_1}^{-1}(u)+F_{X_2}^{-1}(1-u) \text{, \qquad for}\ u\in(0,1),
\end{equation}
with $S^u\overset{d}{=}f(U)$ and $S^l\overset{d}{=}g(U)$. A crucial property underpinning the decompositions of risk measures of comonotonic sums is that the function $f$ is always non-decreasing. In contrast, the function $g$ is not necessarily monotone, which complicates the derivations.

For some classes of marginal distributions, the problem of determining decomposition formulas is relatively simple, and the function $g$ is often monotone under some easily identifiable conditions. For instance, this is the case when $\text{log}(X_i)-\mathbb{E}\left[\text{log}(X_i)\right]$ for $i=1,2$ are symmetric and identically distributed. Another example is when $X_1$ and $X_2$ are such that $X_i \overset{d}{=}a_i + b_i W_i$, for $i=1,2$, where the $W_i$'s are symmetric and identically distributed random variables, which implies that the function $g$ is always monotone, and the sign of its derivative depends upon the sign of $b_1-b_2$. \cite{Chaoubi} derive decompositions of the VaR and TVaR of counter-monotonic sums in the case of non-decreasing functions $g$ \textbf{(Proposition 1)}, as well as in the more challenging case of concave/convex functions $g$ \textbf{(Propositions 2 and 3)}. \textbf{They also show that the function $g$ can be either non-decreasing or concave/convex for several combinations of marginal distributions $F_{X_1}$ and $F_{X_2}$.}

\subsection{The set $E_x$}
In order to derive decomposition formulas for risk measures of counter-monotonic sums that hold for general marginal distributions, a closer analysis has to be performed on the behavior of the function $g$. A key component in the derivations is the set $E_x$ introduced hereafter. Note that throughout the paper, the set $E_x$ is associated with the function $g$, even though it is not explicitly included in the notation.
\begin{definition}\label{DefSetEk}
	For any $x\in\left(x^{\min},x^{\max} \right)$, with $x^{\max}=\sup_{u\in[0,1]}g(u)$ and $x^{\min}=\inf_{u\in[0,1]}g(u)$, and $j=1,...,N_x$, the element $u_{x,j}\in(0,1)$ belongs to $E_x$ if and only if  there exist $a_j, b_j, c_j \in [0,1]$ with $a_j<c_j$ such that one of the following conditions is satisfied:
	\begin{enumerate}
		\item $g(u)<x$ for $u\in (a_j,b_j)$, and $g(u)= x$ for $u\in(b_j,u_{x,j})$, and $g(u)>x$ for $u\in (u_{x,j},c_j)$.
		\item $g(u)>x$ for $u \in (a_j,b_j)$, and $g(u)= x$ for $u\in(b_j,u_{x,j})$, and $g(u)<x$ for $u\in (u_{x,j},c_j)$.

	\end{enumerate}
\end{definition}

Definition \ref{DefSetEk} states that for $u_{x,j}$ to be an element of $E_x$, the function $g$ must cross the level $x$, where the crossing occurs from below under Condition 1 and from above under Condition 2. If $g$ reaches $x$ without crossing it, no elements would be counted in $E_x$. In case $g$ equals $x$ before crossing it, then the element that belongs to $E_x$ is the last one before the crossing occurs. \textbf{Note that $b_j=u_{x,j}$ if the function is strictly increasing in the neighborhood of the crossing. In such cases, there are no points $u$ in $(b_j,u_{x,j})$, and in particular, the equality $g(u)=x$ is not required.}

Figure \ref{Fig:SetEx} provides a general illustration of how the set $E_x$ is defined. For a given $x\in(x^{\min},x^{\max})$, the set $E_x$ contains the elements $u_{x,1}, u_{x,2},u_{x,3},u_{x,4}$ and $ u_{x,5}$. For $u_{x,1}$, $u_{x,2}$ and $u_{x,3}$, the function is locally continuous and strictly monotone. In particular, Condition 1 of Definition \ref{DefSetEk} is satisfied for $u_{x,1}$ and $u_{x,3}$, where the level $x$ is crossed from below, and Condition 2 is satisfied for $u_{x,2}$, where $x$ is crossed from above. The function $g$ is left-continuous and has a jump \textbf{at} $u_{x,4}$. Inequality $g(u_{x,4}+)<x<g(u_{x,4})$ holds, and Condition 2 is satisfied, where $g(u_{x,4}+)$ is the right-limit of $g$ in $u_{x,4}$. Note that $g(u_{x,4})\neq x$. The points $q_1$ and $q_2$ in Figure \ref{Fig:SetEx} do not belong to the set $E_x$. Although $g(q_1)=g(q_2)=x$, the function $g$ does not cross $x$ immediately after these points. The function $g$ is flat on the interval $[q_2,u_{x,5}]$, with $g(u)=x$ for all $u\in [q_2,u_{x,5}]$. However, only the point $u_{x,5}$ from that interval belongs to the set $E_x$, since for this point Condition 1 of Definition \ref{DefSetEk} is satisfied.

\textbf{Note that Definition \ref{DefSetEk} requires the existence of three points $a_j,b_j,c_j\in[0,1]$ with $a_j<b_j<c_j$. For all four points $u_{x,1}$, $u_{x,2}$, $u_{x,3}$ and $u_{x,4}$, the equality $u_{x,j}=b_j$ holds because the function $g$ is strictly monotone at the level of the crossing. In contrast, the identification of the point $b_j$ which is not necessarily equal to $u_{x,j}$ is required when the function $g$ is constant at the level of the crossing, as is the case of the point $u_{x,5}$.}

\begin{figure}[!h]
	\center
	\begin{tikzpicture}
		\draw[->] (0,0) -- (10,0) node[anchor=north]{} ;
		\draw[->] (0,0) -- (0,4) node[anchor=east] {$g$};
		\draw	(0,0) node[anchor=north east] {0}
		(10,0) node[anchor=north] {1}
		(3.5,0) node[anchor=north] {$u$}
		(0,2.5) node[anchor=east]{$x$}
		(0.5235988/pi/0.5,0) node[anchor=north] {$u_{x,1}$}
		(1.04*pi*0.5,0) node[anchor=north] {$u_{x,2}$}
		(4+0.5235988/pi/0.5,0) node[anchor=north] {$u_{x,3}$}
		(5,0) node[anchor=north] {$u_{x,4}$}
		(5.5,0) node[anchor=north] {$q_1$}
		(6.5,0) node[anchor=north] {$q_2$}
		(7.5,0) node[anchor=north] {$u_{x,5}$};
		\draw (0,2) sin (1,3) cos (2,2) sin (3,1) cos (4,2) sin (5,3);
		\draw (5,2) sin (5.5,2.5) cos (6,2) sin (6.5, 1.5);
		\draw (5,2) circle (2pt);
		\filldraw (5,3) circle (2pt);
		\draw[dashed] (5.5,3) -- (5.5,0); 
		\draw[dashed] (6.5,3) -- (6.5,0); 
		\draw[dashed] (7.5,3) -- (7.5,0); 
		\filldraw (6.5,1.5) circle (2pt);
		\draw (6.5,2.5) -- (7.5,2.5);
		\draw (6.5,2.5) circle (2pt);
		\filldraw (7.5,2.5) circle (2pt);
		\draw (7.5,2.5) sin (8.75,3.5) cos (10,4);
		\draw[thick, dashed] (0,2.5) -- (10,2.5);
		\draw[dashed] (0.5235988/pi/0.5,0) -- (0.5235988/pi/0.5,3);
		\draw[dashed] (1.04*pi*0.5,0) -- (1.04*pi*0.5,3);
		\draw[dashed] (4+0.5235988/pi/0.5,0) -- (4+0.5235988/pi/0.5,3);
		\draw[dashed] (5,0) -- (5,3);
	\end{tikzpicture}
	\caption{\footnotesize Illustration of the function $g$ and the set $E_x$, where $E_x=\left\{u_{x,1}, u_{x,2},u_{x,3},u_{x,4},u_{x,5} \right\}$, and $q_1$ and $q_2$ do not belong to $E_x$.  }
	\label{Fig:SetEx}
\end{figure}
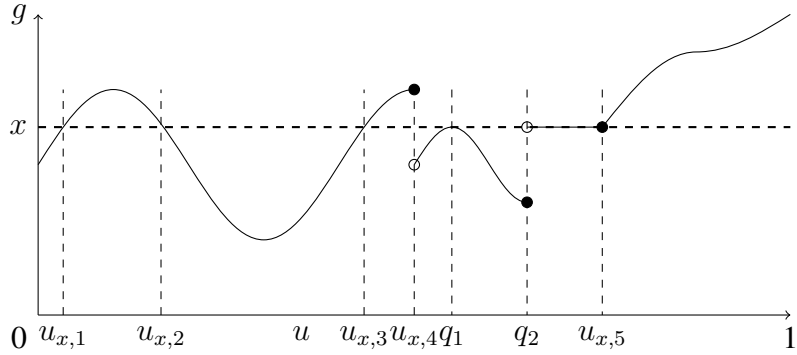

The set $E_x$ is not empty for $x\in\left(x^{\min},x^{\max}\right)$, and is at most countably infinite (i.e.\ $N_x$ can be infinity), as long as $S^l$ is not degenerate. The trivial case where $S^l$ is degenerate is not discussed here. If the function $g$ is monotone, the set $E_x$ contains a single element for any $x\in\left(x^{\min},x^{\max}\right)$. In this case, Condition 1 corresponds to non-decreasingness whereas Condition 2 corresponds to non-increasingness. If $g$ is either convex or concave, the set $E_x$ contains at most two elements, where these two elements are equidistant from the minimum of $g$ when $g$ is symmetric; \textbf{closed-form expressions of the elements of $E_x$ in these cases can be found in Propositions 2 and 3 of \cite{Chaoubi} for symmetric and either convex or concave $g$, respectively, as well as Propositions 6 and 7 where the assumption of symmetry is relaxed}.

A typical situation where the set $E_x$ has multiple elements, and which motivates the present study, is when one of the two random variables, say $X_1$, is continuous (e.g.\ investment fund value) and the other, say $X_2$, is discrete (e.g.\ number of claims). The inverse $F^{-1}_{X_1}$ is strictly increasing whereas the inverse $F^{-1}_{X_2}$ is non-decreasing. As a consequence, the function $g$ increases on the flat segments of $F^{-1}_{X_2}$, and has downward jumps due to the jumps of $F^{-1}_{X_2}$. The situation where $E_x$ has multiple elements can also arise when $X_1$ and $X_2$ are continuous but exhibit skewness and excess kurtosis.


Figure \ref{Figure1} illustrates the function $g$ with two examples. The top panel of Figure \ref{Figure1} displays the function $g$ when $X_1$ and $X_2$ follow a Gamma distribution with shape parameters $4$ and $3$ respectively, and scale parameter $1$ in both cases. The combination of these two distributions results in a concave and asymmetric function $g$. The horizontal line $x$ corresponds to the quantile of the counter-monotonic sum $S^l$ at the level $0.95$, which is equal to $8.94$. Due to the concave shape of $g$, the set $E_{8.94}$ contains two elements. The asymmetry of $g$ implies that the elements of $E_{8.94}$ are not equidistant from the point where the minimum of $g$ is attained.

The bottom panel of Figure \ref{Figure1} displays the case where $X_1$ is assumed to follow a Gamma distribution with shape parameter $5$ and scale parameter $1$, and $X_2$ is assumed to follow a Poisson distribution with rate parameter $5$. The shape of the function $g$ is rather complex. Moreover, for $x=9.85$, which is the median of the counter-monotonic sum, the set $E_{9.85}$ has $12$ elements.
	\begin{figure}[!h]
		\begin{center}
			\includegraphics[scale=0.6]{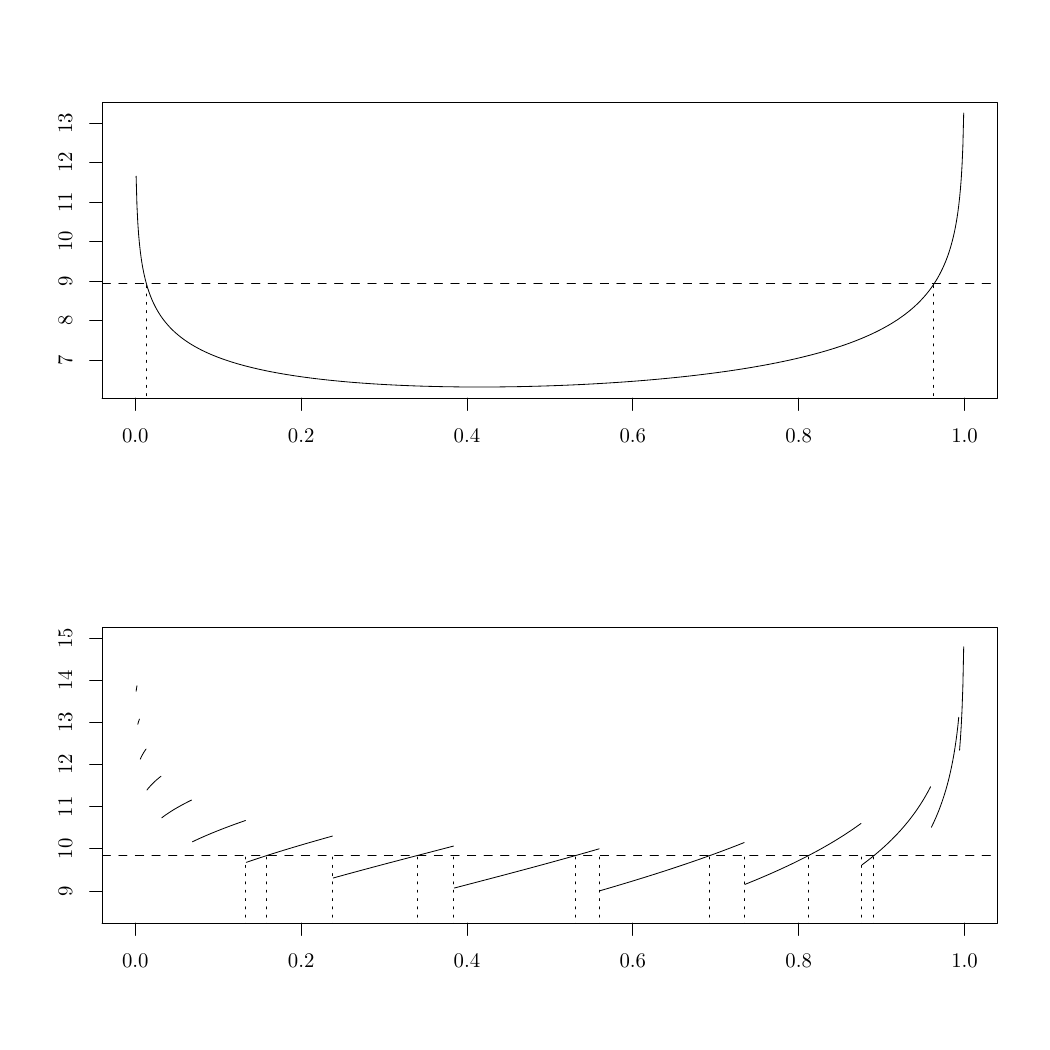}
			\caption{\footnotesize Two examples of the function $g$. In the top panel, both $X_1$ and $X_2$ follow a Gamma distribution with shape parameters $4$ and $3$, respectively, and scale parameter equal to $1$ for both. In the bottom panel, $X_1$ follows a Gamma distribution with shape parameter $5$ and scale parameter $1$, and $X_2$ follows a Poisson distribution with rate parameter $5$.}
			\label{Figure1}
		\end{center}
		
	\end{figure}

The remainder of the paper contains the main results. Namely, the decomposition of the VaR, the TVaR, and the upper tail transform of the counter-monotonic sum are derived. The two examples of the function $g$ from Figure \ref{Figure1} are used to illustrate the results. In the sequel, the notation $u_{x,j}$ and $N_x$ is used to denote the elements and the number of elements of the set $E_x$ when this set is defined for any $x$. When $x$ stands for the inverse $F^{-1}_{S^l}(p)$, the notation $u_{p,j}$ is used instead of $u_{F^{-1}_{S^l}(p),j}$ for the elements of $E_{F^{-1}_{S^l}(p)}$, and the notation $N_p$ is used instead of $N_{F^{-1}_{S^l}(p),j}$. Similarly, $u_{p,j}^\alpha$ and $N_p^{\alpha}$ are used instead of $u_{F^{-1(\alpha)}_{S^l}(p),j}$ and $N_{F^{-1(\alpha)}_{S^l}(p),j}$ when $x$ stands for the generalized inverse $F^{-1(\alpha)}_{S^l}(p)$.

\section{Value-at-Risk of the counter-monotonic sum}\label{Sec:VaR}
\subsection{Main results}
The decomposition of the VaR is closely related to the definition of the set $E_x$. For $p\in (0,1)$, let $u_{p,1}<u_{p,2}<...<u_{p,N_p}$ be the $N_p$ elements of $E_{F^{-1}_{S^l}(p)}$. Recall that by definition, for any $j\in \{1,...,N_p\}$, the function $g$ crosses the level $F^{-1}_{S^l}(p)$ at $u_{p,j}$. If $g$ has a jump \textbf{at} $u_{p,j}$, then $g(u_{p,j})\neq F^{-1}_{S^l}(p)$. If $g$ is continuous \textbf{at} $u_{p,j}$, then $g(u_{p,j})= F^{-1}_{S^l}(p)$. Therefore:
\begin{equation}\label{Eq-6-1}\text{VaR}_p[S^l]=\text{VaR}_{u_{p,j}}[X_1]+\text{VaR}_{1-u_{p,j}}[X_2],\end{equation}
for any $p\in(0,1)$, where $u_{p,j}$ is one of the $N_p$ elements of $E_{F^{-1}_{S^l}(p)}$ in which $g$ does not jump. 

If $g$ is continuous \textbf{at} more than one element of the set $E_{F^{-1}_{S^l}(p)}$, then VaR$_p[S^l]$ admits multiple decompositions, i.e.\ \eqref{Eq-6-1} holds for all elements $u_{p,j}$ in which $g$ does not jump. In fact, the decomposition holds for all $u\in(0,1)$ such that $g(u) = F^{-1}_{S^l}(p)$, and is not restricted to the elements of the set $E_{F^{-1}_{S^l}(p)}$. The decomposition of VaR when $g$ is not monotone is therefore not necessarily unique. 

If $g$ is not continuous \textbf{at} a given element $u_{p,j}$ of $E_{F^{-1}_{S^l}(u_{p,j})}$, then the decomposition \eqref{Eq-6-1} does not hold for that element, i.e.\ $g(u_{p,j})\neq F^{-1}_{S^l}(p)$. Nevertheless, it is still possible to find an $\alpha$-quantile of $X_1$ and a corresponding $(1-\alpha)$-quantile of $X_2$ such that VaR$_p\left[S^l\right]$ is equal to their sum. This result is stated in the following lemma. Note that the decomposition of the VaR of $S^l$ is obtained by setting $x=F^{-1}_{S^l}(p)$.
\begin{lemma}\label{Lemma-4-1}
	For any $x \in\left(x^{\min},x^{\max} \right)$ and any element $u_{x,j}$ in $E_x$ there exists an $\alpha_{x,j}$
	satisfying:
	\begin{equation}\label{Eq_4_1}
	x=F_{X_1}^{-1(\alpha_{x,j})}(u_{x,j})+F_{X_2}^{-1(1-\alpha_{x,j})}(1-u_{x,j}).
	\end{equation}
	Moreover:
	\begin{equation}\label{Eq_4_2}
	\alpha_{x,j}=\left\{
	\begin{array}{ll}
	\frac{x-g(u_{x,j}-)}{g(u_{x,j}+)-g(u_{x,j}-)}, & \hbox{\qquad if } g(u_{x,j}+)\neq g(u_{x,j}-),\\
	0, & \hbox{\qquad if } g(u_{x,j}+)=g(u_{x,j}-).
	\end{array}
	\right.
	\end{equation}
	where $g(u_{x,j}+)$ and $g(u_{x,j}-)$ are the right and left limits, respectively, of the function $g$ in $u_{x,j}$. 
\end{lemma}
\begin{proof}
	For $u_{x,j}\in E_x$, when the function $g$ is continuous \textbf{at} $u_{x,j}$, the equality $g(u_{x,j}-)=g(u_{x,j}+)=x$ holds, and hence:
	$$F_{X_1}^{-1(\alpha_{x,j})}(u_{x,j})+F_{X_2}^{-1(1-\alpha_{x,j})}(1-u_{x,j})=F_{X_1}^{-1}(u_{x,j})+F_{X_2}^{-1}(1-u_{x,j})=x,$$ for any $\alpha_{x,j} \in [0,1]$. Thus, Equation \eqref{Eq_4_1} holds in particular for $\alpha_{x,j}=0$.
	
	When the function $g$ jumps \textbf{at} $u_{x,j}$, the jump size is $|g(u_{x,j}+)-g(u_{x,j}-)|$. Then, there exists $\alpha_{x,j}\in[0,1]$ such that:
	\begin{equation}\label{A1}
		x = \left(1-\alpha_{x,j}\right)g(u_{x,j}-) +  \alpha_{x,j}g(u_{x,j}+).
	\end{equation}
	Note that $F^{-1}_{X_2}\left(1-(u_{x,j}-)\right)=F^{-1}_{X_2}\left(\tilde{u}_{x,j}+\right)$, and $F^{-1}_{X_2}\left(1-(u_{x,j}+)\right)=F^{-1}_{X_2}\left(\tilde{u}_{x,j}-\right)$, where $\tilde{u}_{x,j}=1-u_{x,j}$. Moreover, the inverses $F^{-1}_{X_i}$ are left-continuous, which means that $F^{-1}_{X_i}(q-)=F^{-1}_{X_i}(q)$ and $F^{-1}_{X_i}(q+)=F^{-1+}_{X_i}(q)$ for $q\in(0,1)$. Using \eqref{Eq_2_1}, \eqref{A1} is equivalent to:
	$$x = F^{-1(\alpha_{x,j})}_{X_1}(u_{x,j}) + F^{-1(1-\alpha_{x,j})}_{X_2}(1-u_{x,j}).$$
	Using again \eqref{A1} and rearranging leads to Expression \eqref{Eq_4_2} of $\alpha_{x,j}$ when $g$ has a jump \textbf{at} $u_{x,j}$.
\end{proof}

If $g$ is continuous \textbf{at} $u_{x,j}$, i.e.\ $g(u_{x,j}-)= g(u_{x,j}+)=g(u_{x,j})$, Equality \eqref{Eq_4_1} holds for any choice
of $\alpha_{x,j}$, and more importantly, the equality $g(u_{x,j})=x$ holds. If $g$ has a jump \textbf{at} $u_{x,j}$, i.e.\ $g(u_{x,j}-)\neq g(u_{x,j}+)$, the choice of $\alpha_{x,j}$ is unique. The value of $\alpha_{x,j}$ locates $x$ within the discontinuity, and quantifies the distance between $x$ and the function $g$.

\subsection{Illustration}
Consider the first example where $X_1$ and $X_2$ both follow a Gamma distribution, i.e.\ the top panel on Figure \ref{Figure1}. For $x=F^{-1}_{S^l}\left(0.95\right)$, the set $E_{F^{-1}_{S^l}(0.95)}$ contains the following two elements:
\begin{equation}\label{Eq_6_3}
u_{p,1} = 0.01328 \qquad \text{and} \qquad u_{p,2}= 0.96358.
\end{equation}
Since the function $g$ is continuous, the decomposition in \eqref{Eq-6-1} holds for both $u_{p,1}$ and $u_{p,2}$.

Consider now the second example where $X_1$ follows a Gamma distribution and $X_2$ follows a Poisson distribution, i.e.\ the bottom panel on Figure \ref{Figure1}. For $x=F^{-1}_{S^l}(0.5)$, the set $E_{F^{-1}_{S^l}(0.5)}$ contains the following twelve elements:
\begin{equation}\label{Eq_6_4}
\begin{array}{rrr}
u_{p,1} = 0.13337, &\quad u_{p,2} = 0.15843, &\quad u_{p,3} =  0.23781,\\
u_{p,4} = 0.33971, &\quad u_{p,5} = 0.38404, &\quad u_{p,6} =  0.53079,\\
u_{p,7} = 0.55951, &\quad u_{p,8} = 0.69279, &\quad u_{p,9} =  0.73497,\\ u_{p,10} = 0.81179,&\quad u_{p,11} =  0.87535,&\quad u_{p,12} =  0.89076.
\end{array}
\end{equation}
Figure \ref{Figure1} also shows that the function $g$ jumps \textbf{at} $u_{p,1}$, $u_{p,3}$, $u_{p,5}$, $u_{p,7}$, $u_{p,9}$ and $u_{p,11}$. For these elements, the decomposition \eqref{Eq-6-1} does not hold. Instead, the result of Lemma \ref{Lemma-4-1} has to be used. In particular, from \eqref{Eq_4_2}, it follows that the $\alpha_{p,j}$'s required for \eqref{Eq_4_1} are given by $\alpha_{p,1}=0.83598$, $\alpha_{p,3}=0.46330$, $\alpha_{p,5}=0.22700$, $\alpha_{p,7}=0.16108$, $\alpha_{p,9}=0.31461$ and $\alpha_{p,11}=0.76530$. For the remaining elements, the decomposition \eqref{Eq-6-1} holds.

\section{Tail Value-at-Risk of the counter-monotonic sum}\label{Sec:TVaR}
\subsection{Main results}
The following theorem provides a decomposition for the TVaR of $S^l$ in terms of marginal TVaR's and LTVaR's.
\begin{theorem}\label{Theorem-5-1}
	For any $p\in(0,1)$ and $\alpha\in [0,1]$, let $u_{p,1}^{\alpha}<...<u_{p,N_p^{\alpha}}^{\alpha}$ be the $N_p^{\alpha}$ ordered elements of the set $E_{F^{-1(\alpha)}_{S^l}(p)}$. The \emph{Tail Value-at-Risk} at the level $p$ of $S^{l}$ can be expressed as follows:
	\begin{equation}\label{Eq_5_1}\emph{TVaR}_p\left[S^l\right] = \left\{\begin{array}{ll}\frac{1}{1-p}t_1^\alpha(p),& \qquad \text{if\ \ \ } g(u)\leq F^{-1(\alpha)}_{S^l}(p) \text{\ \ \ for\ \ \ } u\in \left(0,u_{p,1}^{\alpha}\right),\\
	\frac{1}{1-p}t_2^\alpha(p),& \qquad \text{if\ \ \ } g(u)\geq F^{-1(\alpha)}_{S^l}(p) \text{\ \ \ for\ \ \ } u\in\left(0,u_{p,1}^{\alpha}\right),
	\end{array}\right.\end{equation}
	where
\small	\begin{equation}\label{Eq_5_2}\begin{array}{ll}
	t_1^\alpha(p)& = (1-u_{p,1}^{\alpha})\left(\emph{TVaR}_{u_{p,1}^{\alpha}}\left[X_1\right]+\emph{LTVaR}_{1-u_{p,1}^{\alpha}}\left[X_2\right]\right) - \mathcal{T}_{p,N_p^\alpha}^\alpha+F^{-1(\alpha)}_{S^l}(p)\left(u_{p,1}^{\alpha} + \mathcal{D}_{p,N_p^\alpha}^\alpha - p\right),\\
	t_2^\alpha(p)& =u_{p,1}^{\alpha}\left(\emph{LTVaR}_{u_{p,1}^{\alpha}}\left[X_1\right]+\emph{TVaR}_{1-u_{p,1}^{\alpha}}\left[X_2\right]\right) + \mathcal{T}_{p,N_p^\alpha}^\alpha+F^{-1(\alpha)}_{S^l}(p) \left(1-u_{p,1}^{\alpha}-\mathcal{D}_{p,N_p^\alpha}^\alpha-p\right),
	\end{array}\end{equation}\normalsize
	and
	\begin{equation}\label{Eq_5_3}\begin{array}{rl}\mathcal{T}_{p,N_p^\alpha}^\alpha = &\underset{j=2}{\overset{N_p^\alpha}{\sum}}(-1)^j (1-u_{p,j}^{\alpha})\left(\emph{TVaR}_{u_{p,j}^{\alpha}}\left[X_1\right]+\emph{LTVaR}_{1-u_{p,j}^{\alpha}}\left[X_2\right]\right),\\
	\mathcal{D}_{p,N_p^\alpha}^\alpha=&\underset{j=2}{\overset{N_p^\alpha}{\sum}}(-1)^j\left(1-u_{p,j}^{\alpha}\right),
	\end{array}
	\end{equation}
	and $\sum_{j=2}^1=0$ by convention.
\end{theorem}
\begin{proof}
	Using the notation $x_p^{\alpha}=F^{-1(\alpha)}_{S^l}(p)$, where $x_p^\alpha\in (x^{\min},x^{\max})$ for $p\in(0,1)$, the TVaR of the counter-monotonic sum can be written using \eqref{Eq_2_7} as follows:
	\begin{equation}\label{P1}
		\text{TVaR}_p\left[S^l\right]=\frac{1}{1-p}\left(\int_0^1\left(F^{-1}_{S^l}(u)-x_p^{\alpha}\right)_{+}\text{d}u + x_p^{\alpha}\left(1-p\right)\right).
	\end{equation}
	Since $S^l \overset{d}{=}g(U)$, it holds that:
	\begin{equation}\label{P1_1}
		\int_0^1\left(F^{-1}_{S^l}(u)-x_p^{\alpha}\right)_{+}\text{d}u = \int_0^1\left(g(u)-x_p^{\alpha}\right)_{+}\text{d}u.
	\end{equation}
	Let $u_{p,j}^{\alpha}$ for $j=1,...,N_p^{\alpha}$ be the $N_p^{\alpha}$ ordered elements of the set $E_{x_p^{\alpha}}$. Suppose first that $N_p^{\alpha}$ is finite. The following holds:
	\small
	$$
	\int_0^1\left(g(u)-x_p^{\alpha}\right)_{+}\text{d}u =\int_{0}^{u_{p,1}^{\alpha}}\left( g(u)-x_p^{\alpha}\right)
	_{+}\text{d}u+\sum_{j=1}^{N_p^{\alpha}-1}\int_{u_{p,j}^{\alpha}}^{u_{p,j+1}^{\alpha}}\left( g(u)-x_p^{\alpha}\right)
	_{+}\text{d}u+\int_{u_{p,N_p^{\alpha}}^{\alpha}}^{1}\left( g(u)-x_p^{\alpha}\right) _{+}\text{d}u.  
	$$
	\normalsize
	The sign of the function $g(u)-x_p^{\alpha}$ alternates over two consecutive intervals $\left( u_{p,j-1}^{\alpha},u_{p,j}^{\alpha}\right)$ and $\left( u_{p,j}^{\alpha},u_{p,j+1}^{\alpha}\right)$, for $j=1,...,N_p^{\alpha}$, and the sign of $g(u)-x_p^{\alpha}$ on $(0,u_{p,1}^{\alpha})$ determines the sign of $g(u)-x_p^{\alpha}$ on all subsequent intervals. Therefore, depending on whether $N_p^{\alpha}$ is either even or odd, and whether the sign of $g(u)-x_p^{\alpha}$ is either positive or negative on $\left(0,u_{p,1}^{\alpha}\right)$, four possible cases can be distinguished. 
	
	The first case is when $N_p^{\alpha}$ is an even number (i.e.\ there exists an integer $m$ such that $N_p^{\alpha}=2m$), and $g(u)\geq x_p^{\alpha}$ for $u\in\left(0,u_{p,1}^{\alpha}\right)$. In this case, it follows that:%
	\begin{equation}\label{A2}
		\int_0^1\left(g(u)-x_p^{\alpha}\right)_{+}\text{d}u = \int_{0}^{u_{p,1}^{\alpha}}\left(g(u)-x_p^{\alpha}\right) \text{d}%
		u+\sum_{j=1}^{m-1}\int_{u_{p,2j}^{\alpha}}^{u_{p,2j+1}^{\alpha}}\left(g(u)-x_p^{\alpha}\right)
		du+\int_{u_{p,N_p^{\alpha}}^{\alpha}}^{1}\left(g(u)-x_p^{\alpha}\right) \text{d}u. 
	\end{equation}
	The second case is when $N_p^{\alpha}$ is an even number (i.e.\ $N_p^{\alpha}=2m$), and $g(u)\leq x_p^{\alpha}$ for $u\in(0,u_{p,1}^{\alpha})$, which leads to:
	\begin{equation}\label{A3}
		\int_0^1\left(g(u)-x_p^{\alpha}\right)_{+}\text{d}u =\sum_{j=1}^{m}\int_{u_{p,2j-1}^{\alpha}}^{u_{p,2j}^{\alpha}}\left(g(u)-x_p^{\alpha}\right)\text{d}u.  
	\end{equation}
	The third case is when $N_p^{\alpha}$ is an odd number (i.e.\ there exists an integer $m$ such that $N_p^{\alpha}=2m+1$), and $g(u)\geq x_p^{\alpha}$ for $u\in(0,u_{p,1}^{\alpha})$. In this case, it follows that:
	\begin{equation}\label{A4}
		\int_0^1\left(g(u)-x_p^{\alpha}\right)_{+}\text{d}u =\int_{0}^{u_{p,1}^{\alpha}}\left(g(u)-x_p^{\alpha}\right) \text{d}%
		u+\sum_{j=1}^{m}\int_{u_{p,2j}^{\alpha}}^{u_{p,2j+1}^{\alpha}}\left(g(u)-x_p^{\alpha}\right)\text{d}u.
	\end{equation}
	The fourth case is when $N_p^{\alpha}$ is an odd number (i.e.\ $N_p^{\alpha}=2m+1$), and $g(u)\leq x_p^{\alpha}$ for $u\in(0,u_{p,1}^{\alpha})$, which leads to:
	\begin{equation}\label{A5}
		\int_0^1\left(g(u)-x_p^{\alpha}\right)_{+}\text{d}u=\sum_{j=1}^{m}\int_{u_{p,2j-1}^{\alpha}}^{u_{p,2j}^{\alpha}}\left(g(u)-x_p^{\alpha}\right) \text{d}u+\int_{u_{p,N_p^{\alpha}}^{\alpha}}^{1}\left(g(u)-x_p^{\alpha}\right) \text{d}u.
	\end{equation}
	For $j=1,...,N_p^{\alpha}-1$, the following integral split holds:
	\begin{equation}\label{A6}\int_{u_{p,j}}^{u_{p,j+1}}\left(g(u)-x_p\right) \text{d}u=\int_{u_{p,j}}^{1}\left(g(u)-x_p\right) \text{d}u-\int_{u_{p,j+1}}^{1}\left(g(u)-x_p\right)\text{d}u.\end{equation}
	Taking \eqref{A6} into account in \eqref{A2}-\eqref{A5}, the four cases reduce to two only, for any $N_p^{\alpha}$. Namely:
	\begin{equation}\label{P2}
		\int_0^1\left(F^{-1}_{S^l}(u)-x_p^{\alpha}\right)_+\text{d}u =\left\{ \begin{array}{ll}
			I_1^\alpha(p), & \quad \text{if\ \ \ } g(u) \leq x_p^{\alpha} \text{\ \ \ for\ \ \ } u\in (0,u_{p,1}^{\alpha}),\\
			I_2^\alpha(p), & \quad \text{if\ \ \ } g(u) \geq x_p^{\alpha} \text{\ \ \ for\ \ \ } u\in (0,u_{p,1}^{\alpha}),\\
		\end{array}\right.
	\end{equation}
	where
	\begin{eqnarray}
		I_1^\alpha(p)&=& \int_{u_{p,1}^{\alpha}}^1\left(g(u)-x_p^{\alpha}\right)\text{d}u-  \underset{j=2}{\overset{N_p^{\alpha}}{\sum}}\left(-1\right)^{j}\int_{u_{p,j}^{\alpha}}^1\left(g(u)-x_p^{\alpha}\right)\text{d}u, \label{P3}\\
		I_2^\alpha(p)&=& \int_0^{u_{p,1}^{\alpha}}\left(g(u)-x_p^{\alpha}\right)\text{d}u +  \underset{j=2}{\overset{N_p^{\alpha}}{\sum}}\left(-1\right)^{j}\int_{u_{p,j}^{\alpha}}^1\left(g(u)-x_p^{\alpha}\right)\text{d}u.\label{P4}
	\end{eqnarray}

	Taking into account the definitions of the TVaR and LTVaR from \eqref{Eq_2_2} and \eqref{Eq_2_3}, respectively, and using the appropriate change of variable for the integral over $F^{-1}_{X_2}(1-u)$, it follows that:
	\begin{equation}\label{P5}
		\int_{u_{p,j}^{\alpha}}^1\left(g(u)-x_p^{\alpha}\right)\text{d}u = \left(1-u_{p,j}^{\alpha}\right)\left(\text{TVaR}_{u_{p,j}^{\alpha}}\left[X_1\right] + \text{LTVaR}_{1-u_{p,j}^{\alpha}}\left[X_2\right]\right) - x_p^{\alpha}\left(1-u_{p,j}^{\alpha}\right),
	\end{equation}
	for all $j=1,...,N_p^{\alpha}$, as well as:
	\begin{equation}\label{P6}
		\int_0^{u_{p,1}^{\alpha}}\left(g(u)-x_p^{\alpha}\right)\text{d}u = u_{p,1}^{\alpha}\left(\text{LTVaR}_{u_{p,1}^{\alpha}}\left[X_1\right] + \text{TVaR}_{1-u_{p,1}^{\alpha}}\left[X_2\right]\right) - x_p^{\alpha} u_{p,1}^{\alpha}.
	\end{equation}
	Plugging \eqref{P5} and \eqref{P6} in \eqref{P3} and \eqref{P4}, the expressions of $I_1^\alpha(p)$ and $I_2^\alpha(p)$ can be written as follows:
	\begin{eqnarray*}
		I_1^\alpha(p) &=& \left(1-u_{p,1}^{\alpha}\right)\left(\text{TVaR}_{u_{p,1}^{\alpha}}\left[X_1\right] + \text{LTVaR}_{1-u_{p,1}^{\alpha}}\left[X_2\right]\right) - x_p^{\alpha}\left(1-u_{p,1}^{\alpha}-\mathcal{D}_{p,N_p^{\alpha}}^\alpha\right) - \mathcal{T}_{p,N_p^{\alpha}}^\alpha,\\
		I_2^\alpha(p) &=& u_{p,1}^{\alpha}\left(\text{LTVaR}_{u_{p,1}^{\alpha}}\left[X_1\right] + \text{TVaR}_{1-u_{p,1}^{\alpha}}\left[X_2\right]\right) - x_p^{\alpha} \left(u_{p,1}^{\alpha} + \mathcal{D}_{p,N_p^{\alpha}}^\alpha\right) + \mathcal{T}_{p,N_p^{\alpha}}^\alpha,
	\end{eqnarray*}
	where $\mathcal{T}_{p,N_p^\alpha}^\alpha$ and $\mathcal{D}_{p,N_p^\alpha}^\alpha$ are defined in \eqref{Eq_5_3}. 
	Using \eqref{P1}, the TVaR of the counter-comonotonic sum is given by
	\begin{eqnarray*}
		(1-p)\text{TVaR}_p\left[ S^l\right]&=&\left\{ \begin{array}{ll}
			t_1^\alpha(p) , & \quad \text{if\ \ \ } g(u) \leq x_p^{\alpha} \text{\ \ \ for\ \ \ } u\in (0,u_{p,1}^{\alpha}),\\
			t_2^\alpha(p), & \quad \text{if\ \ \ } g(u) \geq x_p^{\alpha} \text{\ \ \ for\ \ \ } u\in (0,u_{p,1}^{\alpha}),\\
		\end{array}\right.
	\end{eqnarray*}
	where $t_1^\alpha(p)=I_1^\alpha(p) + x_p^{\alpha}\left(1-p\right)$ and $t_2^\alpha(p)=I_2^\alpha(p) + x_p^{\alpha}\left(1-p\right)$. The proof ends by noting that combining  $x_p^{\alpha}\left(1-p\right)$ with either $-x_p^{\alpha}\left(1-u_{p,1}^{\alpha}-\mathcal{D}_{p,N_p^\alpha}^\alpha\right)$ or $-x_p^{\alpha}\left(u_{p,1}^{\alpha}+\mathcal{D}_{p,N_p^\alpha}^\alpha\right)$, leads to $ x_p^{\alpha}\left(u_{p,1}^{\alpha} + \mathcal{D}_{p,N_p^\alpha}^\alpha - p\right)$ and $x_p^{\alpha} \left(1-u_{p,1}^{\alpha}-\mathcal{D}_{p,N_p^\alpha}^\alpha-p\right)$, respectively.
\end{proof}

For a given $p\in(0,1)$, each choice of $\alpha\in [0,1]$ will result in a specific decomposition of the Tail Value-at-Risk. In case the marginal cdf's $F_{X_1}$ and $F_{X_2}$ are continuous, all the generalized inverses in \eqref{Eq_5_2} coincide and the decomposition of the TVaR is unique. The choice $\alpha=0$ leads to a decomposition in terms of the usual inverse $F_{S^l}^{-1}(p)$.

The decomposition of the TVaR of the sum in the counter-monotonic case involves more terms than in the comonotonic case. In particular, on top of the linear combination involving the first element $u_{p,1}^{\alpha}$ of the set $E_{F^{-1(\alpha)}_{S^l}(p)}$, the sum $\mathcal{T}_{p,N_p^\alpha}^\alpha$ and the terms $F^{-1(\alpha)}_{S^l}(p)\left(u_{p,1}^{\alpha} + \mathcal{D}_{p,N_p^\alpha}^\alpha - p\right)$ and $F^{-1(\alpha)}_{S^l}(p) \left(1-u_{p,1}^{\alpha}-\mathcal{D}_{p,N_p^\alpha}^\alpha-p\right)$ are also required. 

The number of components of the sum $\mathcal{T}_{p,N_p^\alpha}^\alpha$ depends on the number of elements in the set $E_{F_{S^l}^{-1(\alpha)}(p)}$, and $\mathcal{T}_{p,N_p^\alpha}^\alpha$ vanishes when the function $g$ is monotone. The more complex the shape of $g$ is, the more terms may appear in $\mathcal{T}_{p,N_p^\alpha}^\alpha$. A noteworthy remark is that the notation $\mathcal{T}_{p,N_p^\alpha}^\alpha$ is only used to synthesize the result. This notation does not imply that the contribution of this sum to the TVaR of $S^l$ is less important than the contribution of the  first term containing the marginal TVaR and LTVaR related to the first element $u_{p,1}^{\alpha}$. In fact, as it is shown in the numerical examples below, the sum $\mathcal{T}_{p,N_p^\alpha}^\alpha$ can capture a substantial part of TVaR$_p\left[S^l\right]$.

The terms $F^{-1(\alpha)}_{S^l}(p)\left(u_{p,1}^\alpha + \mathcal{D}_{p,N_p^\alpha}^\alpha - p\right)$ and $F^{-1(\alpha)}_{S^l}(p) \left(1-u_{p,1}^\alpha-\mathcal{D}_{p,N_p^\alpha}^\alpha-p\right)$ in \eqref{Eq_5_2} arise because the quantile function $F^{-1(\alpha)}_{S^l}$ may be constant in the neighborhood of $p$, i.e.\ $F^{-1(\alpha)}_{S^l}$ is `flat' in $p$ and hence, $F_{S^l}\left(F_{S^l}^{-1(\alpha)}(p)\right)$ is not necessarily equal to $p$. For the values of $p$ where $F^{-1(\alpha)}_{S^l}$ is strictly increasing, it can be shown that these terms equal $0$. The following lemma proves this result, thereby providing a link between the elements of the set $E_{F^{-1(\alpha)}_{S^l}(p)}$ and $p$ in case the quantile function $F^{-1(\alpha)}_{S^l}$ is strictly increasing in $p$.
\begin{lemma}\label{Lemma-6-1}
	For any $p\in(0,1)$ and $\alpha \in [0,1]$, suppose that the function $F^{-1(\alpha)}_{S^l}$ is strictly increasing in $p$, and let $u_{p,1}^\alpha<...<u_{p,N_p^\alpha}^\alpha$ be the $N_p^\alpha$ ordered elements of the set $E_{F^{-1(\alpha)}_{S^l}(p)}$. Then:
	\begin{equation}\label{Eq-6-2}p = \left\{\begin{array}{ll}u_{p,1}^\alpha+\mathcal{D}_{p,N_p^\alpha}^\alpha,& \qquad \text{if\ \ \ } g(u)\leq F^{-1(\alpha)}_{S^l}(p) \text{\ \ \ for\ \ \ } u\in (0,u_{p,1}^\alpha),\\
	1-u_{p,1}^\alpha-\mathcal{D}_{p,N_p^\alpha}^\alpha,& \qquad \text{if\ \ \ } g(u)\geq F^{-1(\alpha)}_{S^l}(p) \text{\ \ \ for\ \ \ } u\in(0,u_{p,1}^\alpha),
	\end{array}\right.\end{equation}
	where $\mathcal{D}_{p,N_p^\alpha}^\alpha$ is given in \eqref{Eq_5_3}.
\end{lemma}
\begin{proof}
	For any $p\in (0,1)$ and any $\alpha\in[0,1]$, the quantile $F_{S^l}^{-1(\alpha)}$ is strictly increasing in $p$ if and only if the cdf $F_{S^l}$ is continuous \textbf{at} $x_p^{\alpha}=F_{S^l}^{-1(\alpha)}(p)$. In this case, $F_{S^l}\left( x_p^{\alpha}\right)=p$, and hence, $q\leq p$ is equivalent with $F_{S^l}^{-1(\alpha)}(q) \leq  F_{S^l}^{-1(\alpha)}(p)$. 
	Combining this equivalence with the fact that $S^l \overset{d}{=}g(U)$ and $S^l \overset{d}{=}F^{-1(\alpha)}_{S^l}(U)$ for any $\alpha \in [0,1]$, the TVaR can be written as follows:
	\begin{equation}\label{P8}
		\text{TVaR}_p\left[S^l \right]=\mathbb{E}\left[g(U)\mathbb{I}\left[g(U)
		\geq x_p^\alpha \right] \right]=\frac{1}{1-p}\int_0^1 g(u)\mathbb{I}\left[g(u)\geq x_p^\alpha\right]\text{d}u.
	\end{equation}
	
	Further, the function $F^{-1(\alpha)}_{S^l}$ is strictly increasing in $p\in (0,1)$ if and only if $\mathbb{P}\left[S^l=x_p^{\alpha}\right]=\mathbb{P}\left[g(U)=x_p^{\alpha}\right]=0$,
	and hence, $\mathbb{P}\left[g(U)\geq x_p^{\alpha}\right]=\mathbb{P}\left[g(U) > x_p^{\alpha}\right]$. Therefore, the integral in \eqref{P8} can be decomposed as in \eqref{A5}, where the following holds:
	$$
	\int_{u_{p,j}^\alpha}^1g(u)\text{d}u = \left(1-u_{p,j}^\alpha\right)\left(\text{TVaR}_{u_{p,j}^\alpha}\left[X_1\right] + \text{LTVaR}_{1-u_{p,j}^\alpha}\left[X_2\right]\right),
	$$
	for all $j=1,...,N_p^{\alpha}$, where $u_{p,j}^\alpha$ are the $N_p^{\alpha}$ elements of $E_{x_p^{\alpha}}$, as well as:
	$$
	\int_0^{u_{p,1}^\alpha}g(u)\text{d}u = u_{p,1}^\alpha\left(\text{LTVaR}_{u_{p,1}^\alpha}\left[X_1\right] + \text{TVaR}_{1-u_{p,1}^\alpha}\left[X_2\right]\right).
	$$
	In particular, the terms $F^{-1(\alpha)}_{S^l}(p)\left(u_{p,1}^\alpha + \mathcal{D}^\alpha_{p,N_p^{\alpha}} - p\right)$ and $F^{-1(\alpha)}_{S^l}(p) \left(1-u_{p,1}^\alpha-\mathcal{D}^\alpha_{p,N_p^{\alpha}}-p\right)$ in $t_1^\alpha(p)$ and $t_2^\alpha(p)$, respectively, are equal to $0$, which proves \eqref{Eq-6-2}.
\end{proof}

Lemma \ref{Lemma-6-1} shows that the situation where $F^{-1(\alpha)}_{S^l}$ is strictly increasing in $p$ has two consequences. The first consequence is that the expression of TVaR$_p\left[S^l\right]$ can be simplified because $F^{-1(\alpha)}_{S^l}(p)\left(u_{p,1}^\alpha + \mathcal{D}_{p,N_p^\alpha}^\alpha - p\right)=0$ and $F^{-1(\alpha)}_{S^l}(p) \left(1-u_{p,1}^\alpha-\mathcal{D}_{p,N_p^\alpha}^\alpha-p\right)=0$. The second consequence is that, for $N_p^\alpha\geq 2$, it is sufficient to determine $N_p^\alpha-1$ elements, whereas for $N_p^\alpha=1$, $u_{p,1}^\alpha$ is either equal to $p$ or to $1-p$.

The decomposition of TVaR in Theorem \ref{Theorem-5-1} is further simplified in the next section, where a link between $t_1^\alpha(p)$ and $t_2^\alpha(p)$ is established using the decomposition of the upper tail transform.

\subsection{The case where $N_p^\alpha=1$ and $F_{S^l}^{-1(\alpha)}$ is strictly increasing in $p$}
Suppose that $N_p^\alpha=1$, which includes the case where $g$ is monotone. Additionally, suppose that $F_{S^l}^{-1(\alpha)}$ is strictly increasing in $p$. This situation leads to simple decompositions of TVaR$_p\left[S^l\right]$, which are provided in the following corollary. The result follows directly from Theorem \ref{Theorem-5-1} and Lemma \ref{Lemma-6-1}. In particular, since the generalized quantile function $F_{S^l}^{-1(\alpha)}$ is strictly increasing in $p$, then from Lemma \ref{Lemma-6-1}, $u_{p,1}^\alpha=p$ if $g(0)\leq g(1)$, and $u_{p,1}^\alpha=1-p$ if $g(0)\geq g(1)$\footnote{Note that for $N_p^\alpha=1$, the case $g(u)\leq F^{-1(\alpha)}_{S^l}(p)$ for $u\in (0,u_{p,1}^\alpha)$ is equivalent with $g(0)\leq g(1)$, whereas the case $g(u)\geq F^{-1(\alpha)}_{S^l}(p)$ for $u\in (0,u_{p,1}^\alpha)$ is equivalent with $g(0)\geq g(1)$.}.

\begin{corollary}\label{Corollary-1}
	For any $p\in(0,1)$ and $\alpha \in[0,1]$, if $N_p^\alpha=1$ and the function $F^{-1(\alpha)}_{S^l}$ is strictly increasing in $p$, then:
		
	\begin{equation}\label{Eq_c_1}\emph{TVaR}_p\left[S^l\right] = \left\{\begin{array}{ll}\emph{TVaR}_p[X_1]+\emph{LTVaR}_{1-p}[X_2],& \qquad \text{if\ \ \ } g(0)\leq g(1),\\
			\emph{LTVaR}_{1-p}[X_1]+\emph{TVaR}_p[X_2],& \qquad \text{if\ \ \ } g(0)\geq g(1).
		\end{array}\right.\end{equation}

\end{corollary}
 
The case where $g(0)\leq g(1)$ leads to a decomposition which is consistent with that of \cite{Chaoubi} for strictly increasing functions $g$. \textbf{Specifically, in Proposition 1 of \cite{Chaoubi}, the authors derive the decomposition that follows from $g(0)\leq g(1)$ by assuming that $X_1$ and $X_2$ are continuous random variables ordered in the dispersive order sense.} Thus, the present Corollary \ref{Corollary-1} generalizes their expression by stating that their decomposition does not require the function $g$ to be globally monotone, nor both marginal random variables to be continuous. Indeed, $N_p^\alpha=1$ and local strict increasingness of $F_{S^l}^{-1(\alpha)}$ in $p$ are sufficient.


Interestingly, under these conditions, the dependence uncertainty spread (i.e.\ the difference between the upper bound and the lower bound) can be expressed in terms of the TVaR and LTVaR of only one of the random variables. In particular, using the decomposition of the upper bound TVaR$_p[S^u]$, it follows that if $N_p^\alpha=1$ and $F_{S^l}^{-1(\alpha)}$ is strictly increasing in $p$, then:
\begin{equation}\label{Eq_5_4}\text{TVaR}_p[S^u]-\text{TVaR}_p[S^l] = \left\{\begin{array}{ll}\text{TVaR}_p\left[X_2\right] - \text{LTVaR}_{1-p}\left[X_2\right],& \quad \text{if\ \ \  } g(0)\leq g(1),\\
		\text{TVaR}_p\left[X_1\right] - \text{LTVaR}_{1-p}\left[X_1\right],& \quad \text{if\ \ \  }g(0)\geq g(1).
	\end{array}\right.\end{equation}
Thus, the dependence uncertainty spread of TVaR$_p\left[S^l\right]$ depends on the tail distribution of one of the two random variables only.

Suppose that the function $g$ is strictly monotone, which is a special case of the conditions $N_p^\alpha=1$ and $F_{S^l}^{-1(\alpha)}$ is strictly increasing in $p$. Since both the TVaR and the LTVaR are non-decreasing functions of $p$, it follows from \eqref{Eq_5_4} that the dependence uncertainty spread is a non-decreasing function of the level $p$. This result implies that in case $g$ is strictly monotone, a more conservative risk measurement (i.e.\ higher levels of $p$) leads to a larger dependence uncertainty gap. This is in general not the case, as the dependence uncertainty spread can be a non-monotone function of $p$. This can be seen by combining the decomposition of the upper bound TVaR$_p[S^u]$ in \eqref{Eq_2_10} and that of the lower bound given in Theorem \ref{Theorem-5-1}.

\subsection{Illustration}
The decomposition of the TVaR is illustrated using the examples from Figure \ref{Figure1}, with a particular emphasis on the components of the sum $\mathcal{T}_{p,N_p^\alpha}^\alpha$. The parameter $\alpha$ is set to $0$, i.e.\ $F^{-1(\alpha)}_{S^l}(p)=F^{-1}_{S^l}(p)$, and for simplicity, the superscript $\alpha$ is dropped from all relevant quantities. 

In both cases of Figure \ref{Figure1}, it holds that $g(u)\geq F^{-1}_{S^l}(p)$ for $u\in(0,u_{p,1})$. This means that the decomposition is given by $t_2(p)$. Further, as noted previously, Lemma \ref{Lemma-6-1} leads to $p=1-u_{p,1}-\mathcal{D}_{p,N_p}$, which is due to the fact that $g$ does not have flat parts, and hence, the last term in $t_2(p)$ vanishes. This can also be verified numerically using the values given in \eqref{Eq_6_3} and \eqref{Eq_6_4}. Therefore, in both cases, the expression of TVaR$_p\left[S^l\right]$ is given by:
$$\text{TVaR}_p\left[S^l\right]=u_{p,1}\left(\text{LTVaR}_{u_{p,1}}\left[X_1\right]+\text{TVaR}_{1-u_{p,1}}\left[X_2\right]\right) + \mathcal{T}_{p,N_p},$$
where $ \mathcal{T}_{p,N_p}$ is given in \eqref{Eq_5_3} with $\alpha=0$.

\begin{figure}[!h]
	\begin{center}
		\includegraphics[scale=0.6]{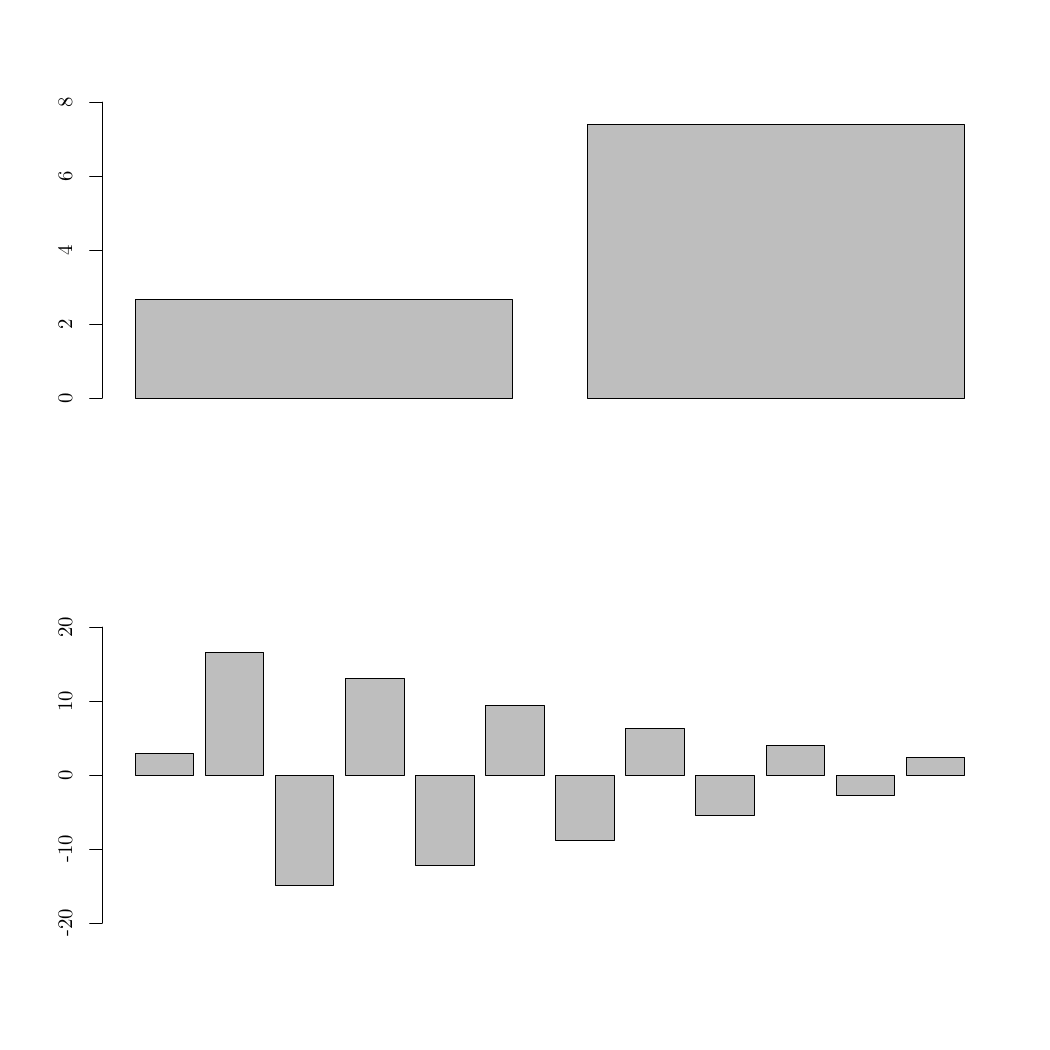}
		\caption{\footnotesize Components of the TVaR decomposition for the two examples of the function $g$ from Figure \ref{Figure1} in the order they appear in the decomposition.}
		\label{Figure2}
	\end{center}
\end{figure}

The set $E_{F^{-1}_{S^l}(0.95)}$ contains two elements in the first example with two Gamma distributed random variables. This means that $N_p=2$, and that $t_2(0.95)$ has two components. The first component is $\frac{u_{p,1}}{1-p}\left(\text{LTVaR}_{u_{p,1}}[X_1]+\text{TVaR}_{1-u_{p,1}}[X_2]\right)$ and the second component is $\mathcal{T}_{p,N_p}=\frac{1-u_{p,2}}{1-p}\left(\text{TVaR}_{u_{p,2}}[X_1]+\text{LTVaR}_{1-u_{p,2}}[X_2]\right)$. Both are displayed on the top panel of Figure \ref{Figure2}, where the first component is on the left and the second component is on the right. The figure shows that $\mathcal{T}_{0.95,N_p}$ is substantially larger than the first component, and largely contributes to the TVaR of the counter-monotonic difference.

In the second example involving Gamma and Poisson random variables, the set $E_{F^{-1}_{S^l}(0.5)}$ contains twelve elements, which means that $t_2(0.5)$ has twelve components. The first component is  $\frac{u_{p,1}}{1-p}\left(\text{LTVaR}_{u_{p,1}}[X_1]+\text{TVaR}_{1-u_{p,1}}[X_2]\right)$. The remaining eleven components with alternating sign are $(-1)^j\frac{1-u_{p,j}}{1-p}\left(\text{TVaR}_{u_{p,j}}[X_1]+\text{LTVaR}_{1-u_{p,j}}[X_2]\right)$, for $j=2,...,12$, where their sum is $\mathcal{T}_{p,N_p}$. These twelve terms are displayed in the bottom panel of Figure \ref{Figure2} in the order in which they appear in the expression, from left to right. Note that the decomposition leads to a TVaR of the counter-monotonic sum equal to $10.51$. In this example, the first component does not contribute as much as the remaining ones. Further, from the second onward, the amplitudes are decreasing. This is due to the fact that the components are weighted by $(1-u_{p,j})$, and hence, additional terms weighted by values of $u_{p,j}$ close to $1$ have less importance.

\textbf{Recall that the TVaR of the comonotonic sum $S^u$ at any level $p\in[0,1]$ satisfies the simple decomposition $\text{TVaR}_p[X_1]+\text{TVaR}_p[X_2]$ in \eqref{Eq_2_10}. Thus, it is worth verifying if the simple decomposition in Corollary \ref{Corollary-1} is an accurate approximation for the TVaR of the counter-monotonic sum $S^l$. Specifically, using the notation $\tilde{t}_1(p)=\text{TVaR}_p[X_1] + \text{LTVaR}_{1-p}[X_2]$ and $\tilde{t}_2(p)=\text{LTVaR}_{1-p}[X_1] + \text{TVaR}_p[X_2]$, the aim now is to verify using the two examples from Figure \ref{Figure1} whether either of the approximations $\text{TVaR}_p[S^l]\approx \tilde{t}_1(p)$ or $\text{TVaR}_p[S^l]\approx \tilde{t}_2(p)$ holds. }

\textbf{Figure \ref{Figure2Add} displays the relative differences (expressed in percentages) between the counter-monotonic TVaR and the approximations $\tilde{t}_1(p)$ (black straight lines) and $\tilde{t}_2(p)$ (blue dashed lines) in function of the level $p$ for both the first example (left panels) and the second example (right panels), where the bottom panels display the corresponding differences in the dependence uncertainty spread. For both examples, it appears that the approximations underestimate the true value of the counter-monotonic TVaR. The approximation $\tilde{t}_1(p)$ underestimates TVaR$_p[S^l]$ by up to $5\%$ for the Gamma-Gamma example and by up to $6\%$ for the Gamma-Poisson example. The second approximation $\tilde{t}_2(p)$ leads to comparatively larger relative differences in the first example with two Gamma distributed random variable, where TVaR$_p[S^l]$ is underestimated by up to $-14\%$. Further, whereas the differences pertaining to the approximation of TVaR$_p[S^l]$ are already substantial, they are exacerbated for the dependence uncertainty spread. Indeed, for both examples, the dependence uncertainty spread would be overestimated by up to $70\%$ in the first example and up to $40\%$ in the second example if one of the approximations were to be used. Therefore, naively approximating TVaR$_p[S^l]$ by the simple decompositions from Corollary \ref{Corollary-1} is not recommended.}

\begin{figure}[!h]
	\begin{center}
		\includegraphics[scale=0.9]{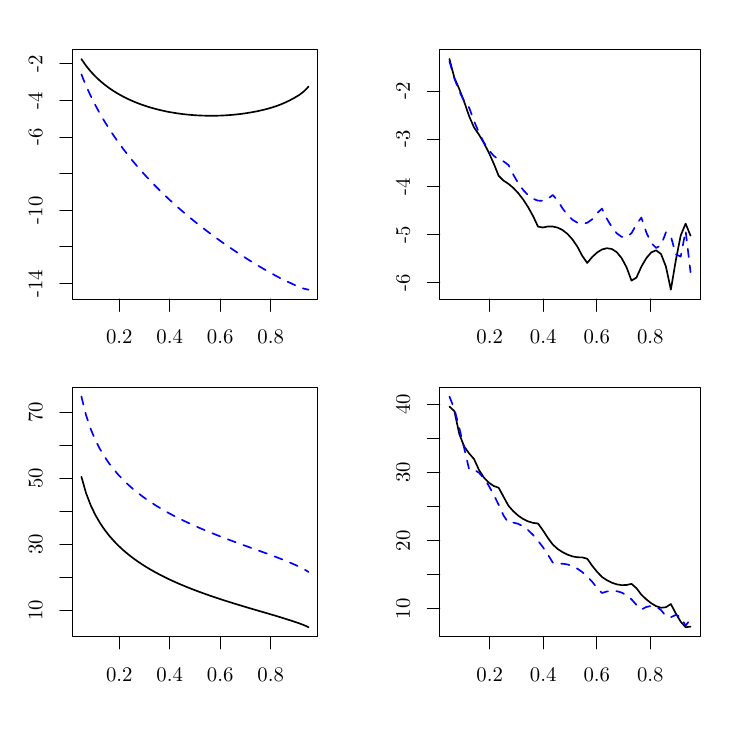}
		\caption{\footnotesize Relative differences (expressed in percentages) between the counter-monotonic TVaR and the approximations $\tilde{t}_1(p)$ (black straight lines) and $\tilde{t}_2(p)$ (blue dashed lines) in function of the level $p$ for both the first example (left panels) and the second example (right panels) displayed on Figure \ref{Figure1}, where the bottom panels display the corresponding differences in the dependence uncertainty spread.}
		\label{Figure2Add}
	\end{center}
\end{figure}

\section{Upper tail transform of the counter-monotonic sum}\label{Sec:SL}
\subsection{Main results}

The following theorem provides a decomposition formula for the upper tail transform of the  counter-monotonic stop-loss premium.
\begin{theorem}\label{Theorem-4-2}
	For any $x\in\left(x^{\min},x^{\max}\right)$, let $u_{x,1}<...<u_{x,N_x}$ be the $N_x$ ordered elements of the set $E_{x}$. The \emph{upper tail transform} at the level $x$ of $S^l$ can be expressed as follows:
	\begin{equation}\label{Eq_4_3x}
		\pi_{S^l}(x)=\left\{
		\begin{array}{c}
			s_1(x), \text{\qquad if\ \ \ }g(u)\leq x\text{\ \ \ for\ \ \ }u\in
			\left(0,u_{x,1}\right),  \\
			s_2(x), \text{\qquad if\ \ \ }g(u)\geq x\text{\ \ \ for\ \ \ }u\in
			\left(0,u_{x,1}\right),
		\end{array}\right.
	\end{equation}
	where:
	\begin{equation}\label{Eq_4_4x}
		\begin{array}{rl}
			s_1(x)  = & \pi_{X_1}\left(F^{-1}_{X_1}\left(u_{x,1}\right)\right)
			-\lambda_{X_2}\left(F^{-1}_{X_2}\left(1-u_{x,1}\right)\right) -\mathcal{S}_{x,N_x} + (1-u_{x,1})\left(g(u_{x,1})-x\right)- \mathcal{J}_{x,N_x},\\
			s_2(x)  = &\pi_{X_2}\left(F^{-1}_{X_2}\left(1-u_{x,1}\right)\right) - \lambda_{X_1}\left(F^{-1}_{X_1}\left(u_{x,1}\right)\right)+\mathcal{S}_{x,N_x} + u_{x,1}\left(g\left(u_{x,1}\right) -x\right) + \mathcal{J}_{x,N_x},\end{array}
	\end{equation}
	with
	\begin{equation}\label{Eq_4_5x}
		\begin{array}{rl}
			\mathcal{S}_{x,N_x}=&\underset{j=2}{\overset{N_x}{\sum}}\left( -1\right) ^{j}\left(\pi_{X_1}\left(F^{-1}_{X_1}(u_{x,j})\right) -\lambda_{X_2}\left(F^{-1}_{X_2}(1-u_{x,j})\right) \right),\\
			\mathcal{J}_{x,N_x} =& \underset{j=2}{\overset{N_x}{\sum}}(-1)^j(1-u_{x,j})\left(g(u_{x,j})-x\right),\end{array}
	\end{equation}
	and $\sum_{j=2}^1=0$ by convention.
\end{theorem}
\begin{proof}
	For $x\in \left(x^{\min},x^{\max}\right)$, there always exist $p\in(0,1)$ and $\alpha\in[0,1]$ such that $x = F_{S^l}^{-1(\alpha)}(p) $. Let $u_{x,1},...,u_{x,N_x}$ be the $N_x$ elements of the set $E_{x}$. Note that $u_{x,j}=u_{p,j}^\alpha$ for all $j=1,...,N_x$, and $N_x=N_p^{\alpha}$. Using the relationships \eqref{Eq_2_7} and \eqref{Eq_2_7x}, it follows that:
	\begin{eqnarray}
		(1-u_{x,j})\left(\text{TVaR}_{u_{x,j}}[X_1]+\text{LTVaR}_{1-u_{x,j}}[X_2]\right) & = &  (1-u_{x,j})g(u_{x,j}) \label{A18}\\
		&+& \pi_{X_1}\left(F^{-1}_{X_1}(u_{x,j})\right)  - \lambda_{X_2}\left(F^{-1}_{X_2}(1-u_{x,j})\right). \nonumber
	\end{eqnarray}
	Define the quantity $A$ as $A=(1-u_{x,1})\left(\text{TVaR}_{u_{x,1}}[X_1]+\text{LTVaR}_{1-u_{x,1}}[X_2]\right)-\mathcal{T}_{p,N_p^{\alpha}}^\alpha$,
	where $\mathcal{T}_{p,N_p^{\alpha}}^\alpha$ is given by \eqref{Eq_5_3}.
	Multiplying Expression \eqref{A18} by $(-1)^j$ and summing over $j$ leads to:
$$A=\pi_{X_1}\left(F^{-1}_{X_1}(u_{x,1})\right)  - \lambda_{X_2}\left(F^{-1}_{X_2}(1-u_{x,1})\right)- \mathcal{S}_{x,N_x} + (1-u_{x,1})g(u_{x,1}) - \underset{j=2}{\overset{N_x}{\sum}}(-1)^j\left(1-u_{x,j}\right)g(u_{x,j}),$$
	where  $\mathcal{S}_{x,N_x}$ is given by \eqref{Eq_4_5x}. By adding $x\left(u_{x,1} + \mathcal{D}_{p,N_p^\alpha}^\alpha - p\right)$ to the left-hand side, where $\mathcal{D}_{p,N_p^\alpha}^\alpha$ is defined in \eqref{Eq_5_3}, it follows that:
	\begin{equation*}\begin{array}{rcl}t_1^\alpha(p)&=&\pi_{X_1}\left(F^{-1}_{X_1}(u_{x,j})\right)  - \lambda_{X_2}\left(F^{-1}_{X_2}(1-u_{x,j})\right)-\mathcal{S}_{x,N_x}+ (1-u_{x,1})g(u_{x,1})\\
		&&  - \underset{j=2}{\overset{N_x}{\sum}}(-1)^j\left(1-u_{x,j}\right)g(u_{x,j}) + x\left(u_{x,1}-1 + \mathcal{D}_{p,N_p^{\alpha}}^\alpha + (1-p)\right),\end{array}
	\end{equation*}
	which can be rearranged such that $t_1^\alpha(p)=s_1(x)   +(1-p)x$. The same reasoning allows to express $t_2^\alpha(p)$ in terms of $s_2(x)$. Therefore, it follows that:
	\begin{equation}\label{A16}
			t_1^\alpha(p)=s_1\left(x\right)+ (1-p)x,\quad \text{and} \quad
			t_2^\alpha(p)=s_2\left(x\right)+(1-p)x,
	\end{equation}
	for any $x \in \left(x^{\min},x^{\max}\right)$. Note also that \eqref{Eq_2_7} immediately leads to $		\pi_{S^{l}}\left(x\right) = \left(1-p\right)\text{TVaR}_p\left[S^{l}\right] - \left(1-p\right)x.$ Combining the latter expression of $\pi_{S^l}(x)$ and \eqref{A16} ends the proof.
\end{proof}

The stop-loss premium of the counter-monotonic sum is equal to the difference of an upper tail transform of one component and a lower tail transform of the other component, but analogously to the decomposition of TVaR$_p\left[S^l\right]$, additional terms appear that depend on the shape of the function $g$. In particular, the sum $\mathcal{S}_{x,N_x}$ accounts for the fact that the function $g$ is not necessarily monotone, and vanishes when $N_x=1$, i.e.\ $\mathcal{S}_{x,1}=0$. The sum $\mathcal{J}_{x,N}$ as well as $(1-u_{x,1})\left(g(u_{x,1})-x\right)$ and $u_{x,1}\left(g(u_{x,1})-x\right)$ appear in case the function $g$ has a jump \textbf{at} one of the elements of $E_{x}$. In general, these terms vanish only when $g$ is continuous \textbf{at} all the elements of $E_x$. This is because if $g$ is continuous \textbf{at} an element $u_{x,j}$, then $g(u_{x,j})=x$.

The following theorem provides an alternative expression for $s_1(x)$ and $s_2(x)$. This alternative decomposition expresses the retentions in terms of the generalized inverses of $X_1$ and $X_2$ instead of left inverses $F_{X_1}^{-1}$ and $F_{X_2}^{-1}$.

\begin{theorem}\label{Theorem-4-1}
	For any $x\in\left(x^{\min},x^{\max}\right)$, let $u_{x,1}<...<u_{x,N_x}$ be the $N_x$ ordered elements of the set $E_{x}$. The \emph{upper tail transform} at the level $x$ of $S^l$ can be expressed as follows:
	\begin{equation}\label{Eq_4_3}
		\pi_{S^l}(x)=\left\{
		\begin{array}{c}
			s_1(x), \text{\qquad if\ \ \ }g(u)\leq x\text{\ \ \ for\ \ \ }u\in
			\left(0,u_{x,1}\right),  \\
			s_2(x), \text{\qquad if\ \ \ }g(u)\geq x\text{\ \ \ for\ \ \ }u\in
			\left(0,u_{x,1}\right),
		\end{array}\right.
	\end{equation}
	where:
	\begin{equation}\label{Eq_4_4}
		\begin{array}{rl}
			s_1(x)  = & \pi_{X_1}\left(F^{-1\left(\alpha_{x,1}\right)}_{X_1}\left(u_{x,1}\right)\right)
			-\lambda_{X_2}\left(F^{-1\left(1-\alpha_{x,1}\right)}_{X_2}\left(1-u_{x,1}\right)\right) -\mathcal{S}_{x,N_x} ,\\
			s_2(x)  = &\pi_{X_2}\left(F^{-1\left(1-\alpha_{x,1}\right)}_{X_2}\left(1-u_{x,1}\right)\right) - \lambda_{X_1}\left(F^{-1\left(\alpha_{x,1}\right)}_{X_1}\left(u_{x,1}\right)\right)+\mathcal{S}_{x,N_x} ,\end{array}
	\end{equation}
	with
	\begin{equation}\label{Eq_4_5}
		\begin{array}{rl}
			\mathcal{S}_{x,N_x}=&\underset{j=2}{\overset{N_x}{\sum}}\left( -1\right) ^{j}\left(\pi_{X_1}\left(F^{-1\left(\alpha_{x,j}\right)}_{X_1}(u_{x,j})\right) -\lambda_{X_2}\left(F^{-1\left(1-\alpha_{x,j}\right)}_{X_2}(1-u_{x,j})\right) \right),
		\end{array}
	\end{equation}
	and $\sum_{j=2}^1=0$ by convention,
	and for $j=1,2,\ldots,N_x$, $\alpha_{x,j}$'s are determined from:
	\begin{equation}\label{Eq_4_6}
		x=F^{-1(\alpha_{x,j})}_{X_{1}}(u_{x,j})+F^{-1(1-\alpha_{x,j})}_{X_{2}}(1-u_{x,j}).
	\end{equation} 
\end{theorem}
\begin{proof}
	The first part of the proof shows that Expressions \eqref{Eq_4_4x} and \eqref{Eq_4_4} of $s_1(x)$ are equal.

	For $j=1,...,N_x$, let $\alpha_{x,j}\in [0,1]$ be such that:
	\begin{equation}\label{Cor-2-eq-9}
		F_{X_1}^{-1(\alpha_{x,j})}\left(u_{x,j} \right)+F_{X_2}^{-1(1-\alpha_{x,j})}\left(1-u_{x,j} \right)=x,
	\end{equation}
	for $x\in(x^{\min},x^{\max})$, where $\alpha_{x,j}$ is known to exist from Lemma \ref{Lemma-4-1}. The stop-loss premium of $X_1$ with retention $F_{X_1}^{-1(\alpha_{x,j})}\left( u_{x,j}\right)$ can be expressed as follows:
	\begin{equation}\label{Cor-2-eq-4}
		\pi_{X_1}\left(F_{X_1}^{-1(\alpha_{x,j})}\left( u_{x,j}\right)\right)=\int_{F_{X_1}^{-1(\alpha_{x,j})}\left( u_{x,j}\right)}^{+\infty} \left(1-F_{X_1}(y)\right)\text{d}y.
	\end{equation}
	For any $\alpha_{x,j}\in [0,1]$, the following equality holds: 
	\begin{equation}\label{Cor-2-eq-5}
		F_{X_1}\left(F_{X_1}^{-1(\alpha_{x,j})}\left( u_{x,j}\right)\right)=F_{X_1}\left(F_{X_1}^{-1}\left( u_{x,j}\right)\right)=u_{x,j}. 
	\end{equation}
	Combining \eqref{Cor-2-eq-4} and \eqref{Cor-2-eq-5} allows to link the stop-loss premiums of $X_1$ with retentions $F_{X_1}^{-1(\alpha_{x,j})}\left( u_{x,j}\right)$ and $F_{X_1}^{-1}\left( u_{x,j}\right)$ as follows:
	\begin{equation}\label{Cor-2-eq-2}
		\pi_{X_1}\left(F_{X_1}^{-1(\alpha_{x,j})}\left( u_{x,j}\right) \right)  =   \pi_{X_1}\left(F_{X_1}^{-1}\left( u_{x,j}\right) \right) - \left(F_{X_1}^{-1(\alpha_x)}\left( u_{x,j}\right) -F_{X_1}^{-1}\left(u_{x,j})\right)  \right)\left( 1-u_{x,j}\right),
	\end{equation}
	Similarly for the lower-tail transform of $X_2$ with retention $F_{X_2}^{-1(1-\alpha_{x,j})}\left( 1-u_{x,j}\right)$:
	\begin{equation}\label{Cor-2-eq-6}
		\lambda_{X_2}\left(F_{X_2}^{-1(1-\alpha_{x,j})}\left(1- u_{x,j}\right)\right)=\int^{F_{X_2}^{-1(1-\alpha_{x,j})}\left( 1- u_{x,j}\right)}_{-\infty}F_{X_2}(y)\text{d}y,
	\end{equation}
	and since the following equality holds:
	\begin{equation}\label{Cor-2-eq-7}
		F_{X_2}\left(F_{X_2}^{-1(1-\alpha_{x,j})}\left( 1-u_{x,j}\right)\right)=F_{X_2}\left(F_{X_2}^{-1}\left( 1-u_{x,j}\right)\right)=1-u_{x,j}, 
	\end{equation}
	for any $\alpha_{x,j}\in [0,1]$, the lower-tail transforms of $X_2$ with retentions $F_{X_2}^{-1(1-\alpha_{x,j})}\left( 1-u_{x,j}\right)$ and $F_{X_2}^{-1}\left(1- u_{x,j}\right)$ can be linked as follows:
	\begin{equation}\label{Cor-2-eq-8}
		\begin{array}{rcl} \lambda_{X_2}\left(F_{X_2}^{-1(1-\alpha_{x,j})}\left(1- u_{x,j}\right) \right)  &=&   \lambda_{X_2}\left(F_{X_2}^{-1}\left(1- u_{x,j}\right) \right)\\ &&+ \left(F_{X_2}^{-1(1-\alpha_{x,j})}\left(1- u_{x,j}\right) -F_{X_2}^{-1}\left( 1-u_{x,j}\right)  \right)\left( 1-u_{x,j}\right).\end{array}
	\end{equation}
	For $j=1,...,N_x$, consider the following difference:
	$$A_j=  \pi_{X_1}\left(F_{X_1}^{-1(\alpha_{x,j})}\left( u_{x,j}\right) \right) -\lambda_{X_2}\left(F_{X_2}^{-1(1-\alpha_{x,j})}\left(1- u_{x,j}\right) \right).$$
	Combining Expressions \eqref{Cor-2-eq-2} and \eqref{Cor-2-eq-8} leads to:
	\begin{eqnarray}
		A_j & = &   \pi_{X_1}\left(F_{X_1}^{-1}\left( u_{x,j}\right) \right) -  \lambda_{X_2}\left(F_{X_2}^{-1}\left(1- u_{x,j}\right) \right) -  \left(F_{X_1}^{-1(\alpha_{x,j})}\left( u_{x,j}\right) -F_{X_1}^{-1}\left(u_{x,j})\right)  \right)\left( 1-u_{x,j}\right) \nonumber\\
		&  & -  \left(F_{X_2}^{-1(1-\alpha_{x,j})}\left(1- u_{x,j}\right) -F_{X_2}^{-1}\left( 1-u_{x,j}\right)  \right)\left( 1-u_{x,j}\right)\nonumber\\
		& = & \pi_{X_1}\left(F_{X_1}^{-1}\left( u_{x,j}\right) \right) -  \lambda_{X_2}\left(F_{X_2}^{-1}\left(1- u_{x,j}\right) \right)+ \nonumber\\
		& & \left( 1-u_{x,j}\right)\left[\left( F_{X_1}^{-1}(u_{x,j})+F_{X_2}^{-1}\left( 1-u_{x,j}\right) \right) - \left(  F_{X_1}^{-1(\alpha_{x,j})}\left(u_{x,j}\right)+F_{X_2}^{-1(1-\alpha_{x,j})}\left( 1-u_{x,j}\right)\right)\right]\nonumber\\
		& =&  \pi_{X_1}\left(F_{X_1}^{-1}\left( u_{x,j}\right) \right) -  \lambda_{X_2}\left(F_{X_2}^{-1}\left(1- u_{x,j}\right) \right) + \left( 1-u_{x,j}\right)\left( g(u_{x,j}) -x \right), \label{Cor-2-eq-11}
	\end{eqnarray} 
	where both the definition of the function $g$ and the relation \eqref{Cor-2-eq-9} were used in the last step. Therefore, \eqref{Cor-2-eq-11} proves that the expressions of $s_1(x)$ given in  \eqref{Eq_4_4x} and \eqref{Eq_4_4} are equal. 
	
	It remains to prove that Expressions \eqref{Eq_4_4x} and \eqref{Eq_4_4} of $s_2(x)$ are equal. 
	Define:
	$$B=  \pi_{X_2}\left(F_{X_2}^{-1(\alpha_{x,1})}\left( u_{x,1}\right) \right) -\lambda_{X_2}\left(F_{X_2}^{-1(1-\alpha_{x,1})}\left(1- u_{x,1}\right) \right).$$ 
	Using the same arguments as above, it follows that:
	\begin{equation}\label{Cor-2-eq-14}
		B=  \pi_{X_2}\left(F_{X_2}^{-1}\left( u_{x,1}\right) \right) -\lambda_{X_2}\left(F_{X_2}^{-1}\left(1- u_{x,1}\right) \right)+u_{x,1}\left(g(u_{x,1})-x \right),
	\end{equation}
	which can be combined with \eqref{Cor-2-eq-11} for $j=2,...,N_x$ to prove that \eqref{Eq_4_4x} and \eqref{Eq_4_4} are two equivalent expressions for $s_2(x)$ as well, which ends the proof.
\end{proof}

In contrast with the expression in Theorem \ref{Theorem-4-2}, the alternative expression in Theorem \ref{Theorem-4-1} does not have the terms $\mathcal{J}_{x,N_x}$ nor $(1-u_{x,1})\left(g(u_{x,1})-x\right)$ and $u_{x,1}\left(g(u_{x,1})-x\right)$. In counterpart, it requires determining the generalized inverses of $X_1$ and $X_2$.

It can be proven that $s_1(x)+s_2(x)=\mathbb{E}\left[S^l\right]-x$. Combining this equality with the parity \eqref{Eq_2_6}, the decomposition of the lower bound $\pi_{S^l}(x)$ can be further simplified. In particular, if $s_1(x)$ is the decomposition of the upper tail transform $\pi_{S^l}(x)$, then $-s_2(x)$ is the decomposition of the corresponding lower tail transform $\lambda_{S^l}(x)$, which means that only one of the two quantities is positive. Hence, as stated in the lemma below, $\pi_{S^l}(x)$ is always the highest of $s_1(x)$ and $s_2(x)$. This simplification applies to the TVaR as well.
\begin{lemma}\label{Corollary-4-1}
	For any $p\in(0,1)$, any $\alpha \in [0,1]$ and any $x\in\left(x^{\min},x^{\max}\right)$, the \emph{Tail Value-at-Risk} at the level $p$ and the upper tail transform at the level $x$ of $S^l$ can be expressed as follows:
	$$\emph{TVaR}_p\left[S^l\right] = \frac{1}{1-p}\max\left\{t_1^\alpha(p),t_2^\alpha(p)\right\},\qquad \text{and}\qquad \pi_{S^l}(x)=\max\{s_1(x),s_2(x)\},$$
	where $t_1^\alpha(p)$ and $t_2^\alpha(p)$ are given in \eqref{Eq_5_2}, and $s_1(x)$ and $s_2(x)$ are given in \eqref{Eq_4_4x}.
\end{lemma}
\begin{proof}
	First, note that using the parity \eqref{Eq_2_6}, it follows that: $$s_1(x)+s_2(x)=\mathbb{E}\left[S^l\right]-x.$$
	Suppose that $\pi_{S^l}(x) =s_1(x) $. Then, from \eqref{Eq_2_6}, $\lambda_{S^l}(x)=-s_2(x)$. Since the lower tail transform is always non-negative, it follows that $s_2(x)\leq0$, and hence, the non-negativeness of the upper tail transform leads to $s_2(x)\leq s_1(x)$. The same reasoning applies when $\pi_{S^l}(x)=s_2(x)$, in which case $s_2(x)\geq s_1(x)$. This concludes the proof for the stop-loss premium.
	
	For $p \in (0,1)$ and $x_p^\alpha=F^{-1}_{S^l}(p)$, it follows that $\pi_{S^l}\left(x_p^\alpha\right) = \max\left\{s_1\left(x_p^\alpha\right),s_2\left(x_p^\alpha\right)\right\}.$ Combining this equality with \eqref{A16} proves that the TVaR of $S^l$ is also the maximum of $\frac{1}{1-p}t_1^\alpha(p)$ and $\frac{1}{1-p}t_2^\alpha(p)$, and hence ends the proof.
\end{proof}
\subsection{The case where $N_x=1$}
For $x\in \left(x^{\min}, x^{\max} \right)$, suppose that the marginal cdf's $F_{X_1}$ and $F_{X_2}$ are such that $N_x=1$, i.e.\ the set $E_x$ contains a single element $u_{x,1}$. The following corollary shows that the decomposition of the stop-loss premium in case $N_x=1$ can be expressed in terms of $F_{S^l}(x)$. Thus, the case $N_x=1$ leads to decompositions which are similar to the decomposition of the upper tail transform of the comonotonic sum in \eqref{Eq_2_11}. Recall that for $N_x=1$, the case $g(u)\leq x$ for $u\in (0,u_{x,1})$ is equivalent with $g(0)\leq g(1)$, whereas the case $g(u)\geq x$ for $u\in (0,u_{x,1})$ is equivalent with $g(0)\geq g(1)$.
\begin{corollary}\label{Corollary_4}
	For any $x\in\left(x^{\min},x^{\max}\right)$, if $N_x=1$, the \emph{upper tail transform} at the level $x$ of $S^l$ can be expressed as follows:
$$	\pi_{S^l}(x)=\left\{
\begin{array}{c}
	\pi_{X_1}\left(F^{-1\left(\alpha_{x}\right)}_{X_1}\left(F_{S^l}(x)\right)\right)
	-\lambda_{X_2}\left(F^{-1\left(1-\alpha_{x}\right)}_{X_2}\left(1-F_{S^l}(x)\right)\right), \text{\quad if\ \ \ }g(0)\leq g(1),  \\
	\pi_{X_2}\left(F^{-1\left(1-\alpha_{x}\right)}_{X_2}\left(F_{S^l}(x)\right)\right) - \lambda_{X_1}\left(F^{-1\left(\alpha_{x}\right)}_{X_1}\left(1-F_{S^l}(x)\right)\right), \text{\quad if\ \ \ }g(0)\geq g(1),
\end{array}\right.$$
where $\alpha_x$ is determined from:
$$	x=\left\{
\begin{array}{c}
	F^{-1\left(\alpha_{x}\right)}_{X_1}\left(F_{S^l}(x)\right) +F^{-1\left(1-\alpha_{x}\right)}_{X_2}\left(1-F_{S^l}(x)\right), \text{\quad if\ \ \ }g(0)\leq g(1),  \\
	F^{-1\left(1-\alpha_{x}\right)}_{X_2}\left(F_{S^l}(x)\right) +F^{-1\left(\alpha_{x}\right)}_{X_1}\left(1-F_{S^l}(x)\right), \text{\quad if\ \ \ } g(0)\geq g(1).
\end{array}\right.$$
\end{corollary}
\begin{proof}
	For $N_x=1$, the set $E_x$ contains a single element $u_{x,1}$ and $\mathcal{S}_{x,1}=0$, and Theorem \ref{Theorem-4-1} leads to a simplified decomposition. The remaining task is to prove that the retentions of the upper and lower tail transforms can be expressed in terms of the cdf of $S^l$. 
	
	Note that $S^l\overset{d}{=}g(U)$ leads to $F_{S^l}(x)   =  \mathbb{P}\left[g(U)\leq x \right]$, and hence, in the first case where $g(u)\leq x$ for $u< u_{x,1}$, it follows that $F_{S^l}(x)= u_{x,1}$.
	
	In the second case where  $g(u)\geq x$ for $u< u_{x,1}$, unlike in the first case, the inequality $F_{S^l}(x)\geq \mathbb{P}\left[g(U)< x\right]=1-u_{x,1}$ holds, and becomes an equality if $g$ is strictly decreasing in $u_{x,1}$, in which case $F_{S^l}(x)   =  1-u_{x,1}$. Otherwise, from the definition of the set $E_x$, the case where the function $g$ is flat before $u_{x,1}$ leads to an inequality, and there exists $u^{\star}<u_{x,1}$ such that $g(u)=x$ for $u\in \left(u^{\star}, u_{x,1} \right)$. Thus, $F_{S^l}(x)=1-u^{\star}$, and the upper tail transform of $S^l$ can be written as follows:
	\begin{equation}\label{Cor4}\mathbb{E}\left[\left(S^l-x\right)_+\right]=\int_0^{u^{\star}} \left(g(u)-x \right)\text{d}u. \end{equation}
	Using Lemma \ref{Lemma-4-1}, there exists $\alpha_x\in [0,1]$ such that:
	\begin{equation}\label{Cor5}x=F^{-1(\alpha_x)}_{X_1}\left(u^{\star}\right) + F^{-1(1-\alpha_x)}_{X_2}\left(1-u^{\star}\right)=F^{-1(\alpha_x)}_{X_1}\left(1-F_{S^l}(x)\right) + F^{-1(1-\alpha_x)}_{X_2}\left(F_{S^l}(x)\right).\end{equation}
	Combining \eqref{Cor4} and \eqref{Cor5} and rearranging ends the proof.
\end{proof}

Note that unlike for the TVaR, the dependence uncertainty spread of $\pi_{S^l}(x)$ in case $N_x=1$ is not determined by one of the two random variables as in \eqref{Eq_5_4}.

\subsection{Illustration}
The importance of each term in $\mathcal{S}_{x,N_x}$ is similar to that of the terms in $\mathcal{T}_{p,N_p^\alpha}^\alpha$, which means that the discussion from the illustration of Section \ref{Sec:TVaR} remains valid. This numerical illustration focuses on the importance of the terms due to the jumps of the function $g$. Only the second example where $X_2$ follows a Poisson distribution is discussed, because there are no jumps in the other example. In order to illustrate the term due to the jumps, the decomposition from Theorem \ref{Theorem-4-2} is used. For $x=F^{-1}_{S^l}(0.5)$, the bottom panel of Figure \ref{Figure1} shows that $g(u)\geq x$ for $u\in(0,u_{x,1})$. Thus, the decomposition of the stop-loss premium is given by $s_2(x)$, and the component under interest is $u_{x,1}\left(g\left(u_{x,1}\right) -x\right) + \mathcal{J}_{x,N_x}$. 

Using the values of the elements of $E_x$, it follows from Theorem \ref{Theorem-4-2} that the stop-loss premium at the level $x=F^{-1}_{S^l}(0.5)$ of $S^l$ is equal to $0.33610$. The first term gives $\pi_{X_2}\left(F^{-1}_{X_2}\left(1-u_{x,1}\right)\right) - \lambda_{X_1}\left(F^{-1}_{X_1}\left(u_{x,1}\right)\right)=0.03918$, whereas $\mathcal{S}_{x,12}=0.92823$. The remaining part due to jumps gives $u_{x,1}\left(g\left(u_{x,1}\right) -x\right) + \mathcal{J}_{x,12}=-0.63130$. This shows the importance of the last term, as ignoring it would considerably overestimate the value of the stop-loss premium.

\section{Conclusion}
\label{Sec:Conclusion}
This paper studies the Value-at-Risk, the Tail Value-at-Risk, and the upper tail transform of the sum of two counter-monotonic random variables $S^l\overset{d}{=}F^{-1}_{X_1}(U)+F^{-1}_{X_2}(1-U)$. Decomposition formulas for these risk measures are derived in terms of the corresponding risk measures of the marginal random variables. An important step in the derivation is to study the behavior of the function $g:u\mapsto F^{-1}_{X_1}(u)+F^{-1}_{X_2}(1-u)$. This is performed by introducing the set $E_x$ which allows to identify the crossing points of $g$ with respect to a level $x$.

The derivations of this paper do not require strong restrictions on the random variables $X_1$ and $X_2$. In particular, these random variables can be continuous, discrete, or a combination of the two. The results are also valid for random variables which are defined on the entire real support, and are not limited to positive-valued random variables. 

The contributions of this paper are relevant to different areas of finance and actuarial science. One of the applications is asset and liability management where differences of random variables are involved rather than sums. In this case, the counter-monotonic sum becomes a comonotonic difference. The study of basis risk is also an area of application where the relevant quantity is a difference of two random variables. 

\textbf{A compelling question is whether more general risk measures of counter-monotonic sums, such as distorted expectations, could be decomposed into their corresponding marginal components. \cite{Cheung2015} provide an expression of distorted expectations in function of the TVaR under some conditions of the distortion function and the random variable. This formula could constitute a starting point of the derivation, but some complications could arise when the random variables $X_1$ and $X_2$ do not satisfy continuity conditions. Such an endeavor deserves a separate contribution, and is left for future research.
}
\bibliographystyle{agsm}
\bibliography{References}

@article{Cheung2015,
	Author = {Cheung, Ka Chun and Dhaene, Jan and Kukush, Alexander and Linders, Daniel},
	Journal = {Scandinavian Actuarial Journal},
	Title = {Ordered random vectors and equality in distribution},
	volume = {2015},
	pages={221--244},
	Year = {2015}}

@article{Zhang2017,
	Author = {Zhang, Jingong and Tan, Ken Seng and Weng, Chengguo},
	Journal = {Insurance: Mathematics and Economics},
	Title = {Optimal hedging with basis risk under mean–variance criterion},
	volume = {75},
	pages={1--15},
	Year = {2017}}

@article{CheungLo2013,
	Author = {Cheung, Ka Chun and Lo, Ambrose},
	Journal = {Insurance: Mathematics and Economics},
	Title = {General lower bounds on convex functionals of aggregate sums},
	volume = {53},
	pages={884--896},
	Year = {2013}}

@article{WangWang2015,
	Author = {Wang, Bin and Wang, Ruodu},
	Journal = {Journal of Multivariate Analysis},
	Title = {Extreme negative dependence and risk aggregation},
	volume = {136},
	pages={12--25},
	Year = {2015}}

@article{GaffkeRuschendorf,
	Author = {Gaffke, N. and R\"{u}schendorf, L},
	Journal = {Mathematische Operationsforscung und Statistik. Series Optimization},
	Title = {On a class of extremal problems in statistics},
	volume = {12},
	pages={123--135},
	Year = {1981}}

@article{WangPengYang2013,
	Author = {Wang, Ruodu and Peng, Liang and Yang, Jingping},
	Journal = {Finance Stochastics},
	Title = {Bounds for the sum of dependent risks and worst value-at-risk with monotone marginal densities},
	volume = {17},
	pages={395--417},
	Year = {2013}
	}

@article{BahlSabanis,
	Author = {Bahl, Raj Kumari and Sabanis, Sotirios},
	Journal = {Insurance: Mathematics and Economics},
	Title = {Model-independent price bounds for catastrophic mortality bonds},
	volume = {96},
	pages={276--291},
	Year = {2021}}

@article{Cheung2017,
	Author = {Cheung, Ka Chun and Denuit, Michel and Dhaene, Jan},
	Journal = {Scandinavian Actuarial Journal},
	Title = {Tail mutual exclusivity and tail-{V}a{R} lower bounds},
	volume = {2017},
	pages={88--104},
	Year = {2017}}

@article{Embrechtal2015,
	Author = {Embrechts, Paul and Wang, Bin and Wang, Ruodu},
	Journal = {Finance and Stochastics},
	Title = {Aggregation-robustness and model uncertainty of regulatory risk measures},
	volume = {19},
	pages={763--790},
	Year = {2015}}

@article{Bernardetal2014,
	Author = {Bernard, C. and Jiang, X. and Wang, R.},
	Journal = {Insurance: Mathematics and Economics},
	Title = {Risk aggregation with dependence uncertainty},
	volume = {54},
	pages={93--108},
	Year = {2014}}

@article{Puccetti2013,
	Author = {Puccetti, Giovanni},
	Journal = {Statistics and Probability Letters},
	Title = {Sharp bounds on the expected shortfall for a sum of dependent random variables},
	volume = {83},
	pages={1227--1232},
	Year = {2013}}

@article{HanbaliLinders2019,
	Author = {Hanbali, Hamza and Linders, Dani\"{e}l},
	Journal = {Quantitative Finance},
	Title = {American-type basket option pricing: a simple two-dimensional partial differential equation},
	volume = {19},
	pages={1689--1704},
	Year = {2019}}

@article{PuccettiWang,
	Author = {Puccetti, Giovanni and Wang, Ruodu},
	Journal = {Statistical Science},
	Title = {Extremal dependence concepts},
	volume = {30},
	pages={485--517},
	Year = {2015}}

@article{WangWang2016,
	Author = {Wang, B and Wang, Ruodu},
	Journal = {Mathematics of Operations Research},
	Title = {Joint mixability},
	volume = {41},
	pages={808--826},
	Year = {2016}}

@book{McNeil,
	author = {McNeil, A. J. and Frey, R. and Embrechts, P.},
	title = {Quantitative risk management: concepts, techniques and tools},
	publisher = {Princeton University Press},
	year = {2015}}

@article{BernardJBF,
	Author = {Bernard, Carole and Vanduffel, Steven},
	Journal = {Journal of Banking and Finance},
	Title = {A new approach to assessing model risk in high dimensions},
	volume = {58},
	pages={166--178},
	Year = {2015}}

@article{BernardJRI,
	Author = {Bernard, Carole and R\"{u}schendorf, Ludger and Vanduffel, Steven},
	Journal = {Journal of Risk and Insurance},
	Title = {Value-at-risk bounds with variance constraints},
	volume = {84},
	pages={923--959},
	Year = {2017}}

@article{Luxa,
	Author = {Luxa, Thibaut and Papapantoleon, Antonis},
	Journal = {Insurance: Mathematics and Economics},
	Title = {Model-free bounds on value-at-risk using extreme value information and statistical distances},
	volume = {86},
	pages={73--83},
	Year = {2019}}

@article{Mesfioui,
	Author = {Mesfioui, Mhamed and Quessy, Jean-Francois},
	Journal = {Insurance: Mathematics and Economics},
	Title = {Bounds on the value-at-risk for the sum of possibly dependent risks},
	volume = {37},
	pages={135--151},
	Year = {2005}}

@article{Embrechts2013b,
	Author = {Embrechts, P. and Puccetti, G. and R\"{u}schendorf, L.},
	Journal = {Journal of Banking and Finance},
	Title = {Model uncertainty and {V}a{R} aggregation},
	volume = {37},
	pages={2750--2764},
	Year = {2013}}

@article{Chaoubi,
	Author = {Chaoubi, Ihsan and Cossette, H\'{e}l\`{e}ne and Gadoury, Simon-Pierre and Marceau, Etienne},
	Journal = {Insurance: Mathematics and Economics},
	Title = {On sums of two counter-monotonic risks},
	volume = {92},
	pages={47--60},
	Year = {2020}}

@book{Wuthrich,
	author = {W\"{u}thrich, Mario V},
	title = {Financial modeling, actuarial valuation and solvency in insurance},
	publisher = {Springer Finance},
	year = {2020}}

@article{Embrechts2009,
	Author = {Embrechts, P. and Lambrigger, D. D. and W\"{u}thrich, M. V.},
	Journal = {Extremes},
	Title = {Multivariate extremes and the aggregation of dependent risks: Examples and counter-examples},
	volume = {12},
	pages={107--127},
	Year = {2009}}

@article{HanbaliLinders,
	Author = {Hanbali, Hamza and Linders, Dani\"{e}l},
	Journal = {University of Amsterdam Research Report},
	Title = {Monotone tail functions: definitions, properties, and applications},
	Year = {2021}}

@article{Kaas_etal:Upper_Lower_Bounds,
	Author = {Kaas, Rob and Dhaene, Jan and Goovaerts, Marc J.},
	Journal = {Insurance: Mathematics and Economics},
	Pages = {151-168},
	Title = {Upper and lower bounds for sums of random variables},
	Volume = {27},
	Year = {2000}}

@article {Dhaene2000,	
	author = {Dhaene, Jan and Wang, Shaun and Young, Virginia and Goovaerts, Marc J.},	
	title = {Comonotonicity and maximal stop-loss premiums},	
	journal = {Bulletin of the Swiss Association of Actuaries},	
	volume = {2000},	
	pages = {99-113},
	year = {2000},
}

@article{Hobson2005,
	author               = {Hobson, David and Laurence, Peter and Wang, Tai-Ho},
	journal              = {Quantitative Finance},
	pages                = {329--342},
	title                = {Static-arbitrage upper bounds for the prices of basket options},
	volume               = {5},
	year                 = {2005},
}

@article {HuntBlake2015,	
author = {Hunt, Andrew and Blake, David},	
title = {Modelling longevity bonds: {A}nalysing the {S}wiss {R}e {K}ortis bond},	
journal = {Insurance: Mathematics and Economics},	
volume = {63},
pages = {12--29},
year = {2015},
}

@article {ChenMacMinnSun2015,	
author = {Chen, Hua and MacMinn, Richard and Sun, Tao},	
title = {Multi-population mortality models: {A} factor copula approach},	
journal = {Insurance: Mathematics and Economics},	
volume = {63},
pages = {135--146},
year = {2015a},
doi = {10.1016/j.insmatheco.2015.03.022}
}

@article{Dhaene_etal_2006,
	Author = {Dhaene, J and Vanduffel, S and Goovaerts, M and Kaas, R and Tang, Q and Vyncke, D},
	Journal = {Stochastic {M}odels},
	Pages = {573--606},
	Title = {Risk measures and comonotonicity: a review},
	Volume = {22},
	Year = {2006},}

@article{Dhaene_et_al_2005,
	Author = {J. Dhaene and S. Vanduffel and M. J. Goovaerts and R. Kaas and D. Vyncke},
	Journal = {Journal of Risk and Insurance},
	Pages = {253--301},
	Title = {Comonotonic approximations for optimal portfolio selection problems},
	Volume = {72},
	Year = {2005}}

@article {LaurenceWang2009,
author = {Laurence, Peter and Wang, Tai-Ho},	
title = {Sharp distribution free lower bounds for spread options and the corresponding optimal subreplicating portfolios},	
journal = {Insurance: Mathematics and Economics},	
volume = {44},	
pages = {35-47},
year = {2009},
}

@article {LaurenceWang2008,	
author = {Laurence, Peter and Wang, Tai-Ho},	
title = {Distribution free bounds for spread options and market implied antimonotonicity gap},	
journal = {European Journal of Finance},	
volume = {14},	
pages = {717-734},
year = {2008},
}

@article {Coughlaan2011,
author = {Coughlan, Guy D. and Khalaf-Allah, Marwa and Ye, Yijing and Kumar, Sumit and Cairns, Andrew J. and Blake, David and Dowd, Kevin},
title = {Longevity hedging 101},
journal = {North American Actuarial Journal},
volume = {15},
pages = {150--176},
year = {2011},
}

@article {DenuitDhaene2007,	
author = {Denuit, Michel and Dhaene, Jan},	
title = {Comonotonic bounds on the survival probabilities in the {L}ee-{C}arter model for mortality projections},	
journal = {Computational and Applied Mathematics},	
volume = {203},	
pages = {169--176},
year = {2007},
}

@article {Linders2012,	
author = {Linders, Dani\"{e}l and Dhaene, Jan and Hounnon, H and Vanmaele, Mich\`{e}le},	
title = {Index Options: A Model Free Approach},	
journal = {FEB, KU Leuven Research Report},	
volume = {AFI-1265},
year = {2012},
}

@article {CarmonaDurrleman2003,	
author = {Carmona, Ren\'{e} and Durrleman, Valdo},	
title = {Pricing and Hedging Spread Options},	
journal = {SIAM Review},	
volume = {45},
pages = {627--685},
year = {2003},
}

@article {Dhaene2002a,	
author = {Dhaene, Jan and Denuit, Michel and Goovaerts, Marc J. and Kaas, Rob and Vyncke, David},	
title = {The concept of comonotonicity in actuarial science and finance: {T}heory},	
journal = {Insurance: Mathematics and Economics},	
volume = {31},	
pages = {3--33},
year = {2002a},
}

@article {Dhaene2002b,	
author = {Dhaene, Jan and Denuit, Michel and Goovaerts, Marc J. and Kaas, Rob and Vyncke, David},	
title = {The concept of comonotonicity in actuarial science and finance: {A}pplication},	
journal = {Insurance: Mathematics and Economics},	
volume = {31},	
pages = {133--161},
year = {2002b},
}

@article {MeileisonNadas1979,	
author = {Meileison, Isaac and N\'{a}das, Arthur},	
title = {Convex majorization with an application to the length of critical paths},	
journal = {Journal of Applied Probability},	
volume = {16},
pages = {671--677},
year = {1979},
}

@article{Ruschendorf1983,
	author = {R\"uschendorf, Ludger},
	title = {Solution of a statistical optimization problem by rearrangement methods},
	journal = {Metrika},
	volume = {30},
	year = {1983},
	pages = {55-61},
}

@book {DenuitDhaeneGoovaertsKaas2005,	
author = {Denuit, Michel and Dhaene, Jan and Goovaerts, Marc and Kaas, Rob},	
title = {Actuarial Theory for Dependent Risks},	
Publisher = {John Wiley \& Sons},	
year = {2005},
}

@article {DhaeneGoovaerts,	
author = {Dhaene, Jan and Goovaerts, Marc},	
title = {Dependency of risks and stop-loss order},	
journal = {ASTIN Bulletin},	
volume = {26},	
pages = {201--2012},
year = {1996},
}

@article{FengJingDhaene,
author = {Feng, R and Jing, X and Dhaene, J},
title = {Comonotonic approximations of risk measures for variable annuity guaranteed benefits with dynamic policyholder behavior},
journal = {Journal of Computational and Applied Mathematics},
volume = {311},
pages = {272--292},
year = {2015},
}

@article {Embrechts2013a,
	author = {Embrechts, Paul and Hofert, Marius},
	title = {A note on generalized inverses},
	journal = {Mathematical Methods of Operations Research},
	volume = {77},
	pages = {423-432},
	year = {2013},
}

@book {ShakedShanthihumar,	
	author = {Shaked, Moshe and Shanthikumar, Goerge},	
	title = {Stochastic Orders},	
	Publisher = {Springer},	
	year = {2007},
}

@inproceedings{Deelstra2010:overview_Comonotonicity,
	author               = {Griselda Deelstra and Jan Dhaene and Michele Vanmaele},
	booktitle            = {Advanced Mathematical Methods for Finance},
	editor               = {Oksendal, B. and Nunno, G.},
	pages                = {155-179},
	publisher            = {Springer Berlin Heidelberg},
	title                = {An overview of comonotonicity and its applications in finance and insurance},
	year                 = {2011},
}

@article{Dhaene_Denuit_1999,
	Author = {Jan Dhaene and Michel Denuit},
	Journal = {Insurance: Mathematics and Economics},
	Pages = {11 - 21},
	Title = {The safest dependence structure among risks},
	Volume = {25},
	Year = {1999},}

@article{Wang_wANG_2011cm,
	Author = {Bin Wang and Ruodu Wang},
	Journal = {Journal of Multivariate Analysis},
	Pages = {1344 - 1360},
	Title = {The complete mixability and convex minimization problems with monotone marginal densities},
	Volume = {102},
	Year = {2011},}

@article{LeeAhn_2014,
	Author = {Lee, Woojoo and Ahn, Jae Youn},
	Journal = {Insurance: Mathematics and Economics},
	Pages = {68--79},
	Title = {On the multidimensional extension of countermonotonicity and its applications},
	Volume = {56},
	Year = {2014},}

@article{Cheung_Lo_2014,
	Author = {Ka Chun Cheung and Ambrose Lo},
	Journal = {Insurance: Mathematics and Economics},
	Pages = {180 - 190},
	Title = {Characterizing mutual exclusivity as the strongest negative multivariate dependence structure},
	Volume = {55},
	Year = {2014},}

@article{Chen2015,
	author = {Chen, Xinliang and Deelstra, Griselda and Dhaene, Jan and Linders, Daniel and Vanmaele, Michel},
	title = {On an optimization problem related to static super-replicating strategies},
	journal	= {Journal of Computational and Applied Mathematics},
	volume = {278},
	pages ={213-230},
	year = {2015b},
}

@article {Chen2008,	
	author = {Chen, Xinliang and Deelstra, Griselda and Dhaene, Jan and Vanmaele, Mich\`{e}le},	
	title = {Static Super-Replicating Strategies for a Class of Exotic Options},	
	journal = {Insurance: Mathematics and Economics},	
	volume = {42},	
	pages = {1067--1085},
	year = {2008},
}

@article{Bignozzi_et_al_2015,
	author = {Bignozzi, V. and Puccetti, G. and and L. R\"uschendorf},
	year		=	{ 2015},
	title 		=	{Reducing model risk via positive and negative dependence assumptions},
	journal	= { Insurance Mathematics \& Economics}, 
	volume	= {61},
	pages		= {17-26},
}

\end{document}